\documentclass[12pt]{article}
\usepackage{amsmath,amsthm,amssymb}
\usepackage{amsfonts}
\usepackage{graphicx,epstopdf,framed}
\usepackage{latexsym}
\usepackage{mathrsfs}
\usepackage{pifont, geometry} 
\usepackage{color,enumerate}

\usepackage{etoolbox} 
\apptocmd{\thebibliography}{\setlength{\itemsep}{2pt}}{}{}

\newtheorem{theorem}{Theorem}[section]
\newtheorem{corollary}[theorem]{Corollary}
\newtheorem{lemma}[theorem]{Lemma}
\newtheorem{assumption}[theorem]{Assumption}
\newtheorem{proposition}[theorem]{Proposition}

\theoremstyle{definition}
\newtheorem{definition}[theorem]{Definition}
\newtheorem{remark}[theorem]{Remark}

\newtheorem{example}[theorem]{Example}

\numberwithin{equation}{section}

\newcommand{\BE}{\begin{equation}}
\newcommand{\BEQ}[1]{\BE\label{#1}} 
\newcommand{\rfb}[1]{\mbox{\rm
   (\ref{#1})}\ifx\undefined\stillediting\else:\fbox{$#1$}\fi}

\newfont{\roma}{cmr10 scaled 1200}
\renewcommand{\cline}{{\mathbb C}}
\newcommand{\rline}  {{\mathbb R}}
\newcommand{\CCC}  {{\bf C}}
\newcommand{\GGG}  {{\bf G}}
\newcommand{\PPP}  {{\bf P}}
\newcommand{\dd}   {{\rm d}\hbox{\hskip 0.5pt}}
\newcommand{\Ascr} {{\cal A}}
\newcommand{\Bscr} {{\cal B}}
\newcommand{\Cscr} {{\cal C}}
\newcommand{\Dscr} {{\cal D}}
\newcommand{\Fscr} {{\cal F}}

\newcommand{\Jscr} {{\cal J}}

\newcommand{\Mscr} {{\cal M}}
\newcommand{\Oscr} {{\cal O}}

\newcommand{\Wscr} {{\cal W}}

\newcommand{\mm}    {{\hbox{\hskip 0.5pt}}}
\newcommand{\m}     {{\hbox{\hskip 1pt}}}

\newcommand{\bluff} {{\hbox{\raise 15pt \hbox{\mm}}}}
\newcommand{\sbluff}{{\hbox{\raise  7pt \hbox{\mm}}}}
\renewcommand{\L}    {{\Lambda}}
\renewcommand{\l}    {{\lambda}}
\newcommand{\G}      {{\Gamma}}
\newcommand{\g}      {{\gamma}}
\newcommand{\FORALL} {{\hbox{$\hskip 11mm \forall \;$}}}
\newcommand{\e}      {{\varepsilon}}

\newcommand{\rarrow} {{\,\rightarrow\,}}

\renewcommand{\Re}   {{\rm Re\,}}

\newcommand{\Ker}    {{\rm Ker\,}}

\newcommand{\bbm}[1]{\left[\begin{matrix} #1 \end{matrix}\right]}
\newcommand{\sbm}[1]{\left[\begin{smallmatrix} #1
   \end{smallmatrix}\right]}
\begin{document}
\renewcommand{\thefootnote}{\fnsymbol{footnote}}
\renewcommand{\thefootnote}{\fnsymbol{footnote}}
\newcommand{\footremember}[2]{%
   \footnote{#2}
    \newcounter{#1}
    \setcounter{#1}{\value{footnote}}%
}
\newcommand{\footrecall}[1]{%
    \footnotemark[\value{#1}]%
}
\makeatletter
\def\blfootnote{\gdef\@thefnmark{}\@footnotetext}
\makeatother

\begin{center}
{\Large \bf Minimal order controllers for output\\[1ex]
            regulation of nonlinear systems}\\[2ex]
Vivek Natarajan and George Weiss
\blfootnote{This work was partially supported by grant no. 800/14
of the Israel Science Foundation.}
\blfootnote{V. Natarajan (vivek.natarajan@iitb.ac.in) is with the Systems
and Control Engineering Group, Indian Institute of Technology
Bombay, Mumbai, India, 400076, Ph:+91 2225765385}
\blfootnote{G. Weiss (gweiss@eng.tau.ac.il) is with the School of
Electrical Engineering, Tel Aviv University, Ramat Aviv, Israel,
69978, Ph:+972 36405164.}
\end{center}

\begin{abstract} {\small\vbox{\noindent
This paper is about the nonlinear local error feedback regulator
problem. The plant is a nonlinear finite-dimensional system with a
single control input and a single output and it is locally
exponentially stable around the origin. The plant is driven, via a
separate disturbance input, by a Lyapunov stable exosystem whose
states are nonwandering. The reference signal that the plant output
must track is a nonlinear function of the exosystem state. The local
error feedback regulator problem is to design a dynamic feedback
controller, with the tracking error as its input, such that (i) the
closed-loop system of the plant and the controller is locally
exponentially stable, and (ii) the tracking error tends to zero for
all sufficiently small initial conditions of the plant, the controller
and the exosystem. Under the assumption that the above regulator
problem is solvable, we propose a nonlinear controller whose order is
relatively small - typically equal to the order of the exosystem, and
which solves the regulator problem. The emphasis is on the low order
of the controller. In contrast, previous results on the regulator
problem have typically proposed controllers of a much larger order.
The stability assumption on the plant (which can be relaxed to some
extent) is crucial for making it possible to design a low order
controller. We will show, under certain assumptions, that our proposed
controller is of minimal order. Three examples are presented - the
first illustrates our controller design procedure using an exosystem
whose trajectories are periodic even though the state operator of the
linearized exosystem contains a nontrivial Jordan block. The second
example is more involved, and shows that sometimes a nontrivial
immersion of the exosystem is needed in the design. The third example,
based on output voltage regulation for a boost power converter, shows
how the regulator equations may reduce to a first order PDE with no
given boundary conditions, but which nevertheless has a locally unique
solution.}}\vspace{-3mm}
\end{abstract}

\vspace{-6mm}

\section{Introduction} \label{sec1} \vspace{-1mm}
\setcounter{equation}{0} 

\ \ \ Given a stabilizable plant and an unstable exosystem, the {\em
local error feedback regulator problem} is to design a locally
stabilizing controller (for the plant) which guarantees the tracking
of certain reference signals by the plant output, even when the plant
is driven by external disturbance signals. The reference and
disturbance signals are both functions of the exosystem state, assumed
to be sufficiently small. The input to the controller is the tracking
error and its output is the control input to the plant. The robust
version of the regulator problem is to design a controller that
guarantees both stabilization and tracking for a set of uncertain
plants.

The regulator problem, and its robust version, for linear
finite-dimensional plants and exosystems was addressed in
\cite{FrWo:75} and \cite{Fra:77} using geometric methods. The
solvability of the problem was characterized in terms of the
solvability of certain matrix equations, known in the literature as
the regulator equations. The {\em internal model principle}, which
states that for robust regulation the dynamic structure of the
exosystem (suitably duplicated) must be incorporated into the
controller, was introduced in \cite{FrWo:75}. Following these works,
the regulator problem for nonlinear finite-dimensional plants and
exosystems was addressed in \cite{IsBy:90} in a local setting,
i.e. the proposed controller ensured that its closed-loop system with
the plant is locally exponentially stable and that tracking is
achieved for sufficiently small initial conditions of the plant, the
controller and the exosystem. Nonlinear regulator equations, a
generalization of the regulator equations in \cite{Fra:77}, were
introduced in \cite{IsBy:90}.

Since the early 1990's many researchers have extended the results in
\cite{IsBy:90}, mainly by developing controllers that solve the
nonlinear robust regulator problem in local (\cite{Pri:93, Pri:97,
ByPrIsKa:97}), semi-global (\cite{Kha:94,Isi:97,SeIsMa:00}) and global
(\cite{SeIs:00,ChHu:05,XiDi:07}) settings. Most of these works
(excluding \cite{Kha:94}) address the robust regulator problem using
the regulator equations. The semi-global and global results mentioned
above are obtained for plants that are single-input single-output
(SISO) from control input to output; the recent works \cite{PiHu:11,
IsMa:12, AsIsMaPr:13} consider nonlinear multi-input multi-output
plants. All the aforementioned nonlocal results are derived for plants
that possess additional properties (such as minimum-phase, feedback
linearizable, etc). Another global result is in \cite{JaWe:09}, where
a rather restrictive regulator problem was solved assuming that the
plant is passive. Related results for fully actuated mechanical
systems with unknown parameters are in \cite{JaWe:08}. Using ideas
from the nonlinear regulator theory, a novel method for the
stabilization and adaptive control of nonlinear systems, called
immersion and invariance, was developed in \cite{AsOr:03}. For a
comprehensive introduction to the nonlinear regulator problem, see
\cite{ByPrIs:97,Hua:04,Isi:95}.

In this paper we focus on {\em finding a minimal order (or at least,
low order) controller that solves the nonlinear error feedback
regulator problem in a local setting}.  Specifically, assuming a
locally stable nonlinear plant that is SISO from control input to
output (but it may also have disturbance inputs) and a Lyapunov
stable nonlinear exosystem whose states are nonwandering, connected
as shown in Figure 2, we will construct controllers whose order is
the same as that of a certain immersion of the exosystem. In many
applications, this immersion is just a copy of the exosystem. The
order of the controllers in the papers mentioned earlier, even in
the absence of uncertain plant parameters, is typically equal to or
larger than that of the exosystem and the plant combined. A key reason
for this, apart from the fact that the control objectives in those
papers are different, is the sequential control design approach they
adopt. In their approach, first an internal model (whose order can be
larger than that of the exosystem) is designed and then the loop
containing the internal model and the plant is stabilized using an
additional controller. In contrast we will design a stabilizing
internal model directly. The search for minimal order controllers is
of practical value from an implementation standpoint, and is also of
theoretical interest. For a discussion on the lower bound for the
order of any controller that solves the linear robust regulator
problem see \cite{DaGo:75,DeWa:78}.

Our work is motivated by an alternate approach to the linear regulator
problem, first proposed in \cite{Dav:76} for stable finite-dimensional
linear plants. In this approach, the control and observation operators
of the internal model (whose order is $r\times m$, where $r$ is the
number of outputs and $m$ is the exosystem order) are chosen to ensure
that its closed-loop system with the plant is stable. Hence no
additional stabilizing controller is required and this may lead to the
construction of minimal order robust controllers for the linear
regulator problem.. This approach was extended in \cite{HaPo:00,
ReWe:03} to construct finite-dimensional controllers that solve the
regulator problem for stable linear infinite-dimensional plants. We
will extend this approach to nonlinear finite-dimensional plants that
are SISO from control input to output. Preliminary versions of this
work have appeared in \cite{NaWe:14,WeNa:14}, see Remark \ref{CDC}.

In Section \ref{sec2} we introduce the nonlinear error feedback
regulator problem and present relevant background information.  This
includes Definitions \ref{autonomous} and \ref{immersion} which
describe the notions of autonomous systems and their local immersions.
We then state Theorem \ref{IsBy}, a well-known result about the
connection between the solvability of the local error feedback
regulator problem and the local solvability of the regulator
equations.  This theorem is somewhat more general and precise than
other versions in the literature, see Remark
\ref{Putin_occupies_Crimea}. Finally we state Proposition \ref{ReWe}
about the linear error feedback regulator problem, which follows from
the results in \cite{HaPo:00}.

In Section \ref{sec3}, which is for technical preparation, we state a
version of the classical center manifold theorem. Using this theorem
we prove Proposition \ref{Haragus} which shows that for certain
autonomous systems, given a local immersion, there exists a possibly
lower order local immersion with special spectral properties.

Section \ref{sec4} contains our main result, Theorem \ref{newmain},
which states the existence of a low order controller under some
natural assumptions. This controller contains a suitable local
immersion of an autonomous system constructed using the exosystem and
a solution of the regulator equations. The order of the controller is
equal to the order of the immersion. The input map of the controller
is linear. Under some detectability assumptions on the linearized
plant and linearized exosystem, this controller is of minimal order,
see Corollary \ref{nlDcon}.

Section \ref{sec5} contains three examples. In the first example we
present a nonlinear exosystem, the state operator of whose
linearization at the origin contains a non-trivial Jordan block.  We
show that this exosystem is Lyapunov stable and all its trajectories
are periodic. We then illustrate our controller design using this
exosystem and a second-order plant. The second example shows that in
some cases there must exist a local immersion for the exosystem, with
order larger than that of the exosystem, for the regulator problem to
have a solution. We then present an immersion such that the controller
constructed using it is of minimal order. The third example is output
voltage regulation for the boost power converter shown in
Figure\m\m\m\-1. A constant but unknown input voltage $v>0$ is
transformed into a higher voltage $z_1$ that feeds a load $R$. Due to
the fast switching, there will be high frequency ripple on $z_1$,
which becomes negligible for very high switching frequency, and we
neglect this ripple. The control problem is to make $z_1$ track a
reference value, in spite of the sinusoidal disturbance current
$i_e$. After putting this system into our framework, we shall see that
the regulator equations reduce to a first order PDE with no given
boundary conditions, but which nevertheless has a locally unique
solution. \vspace{-6.5mm}

\begin{center}
\includegraphics[scale=0.65]{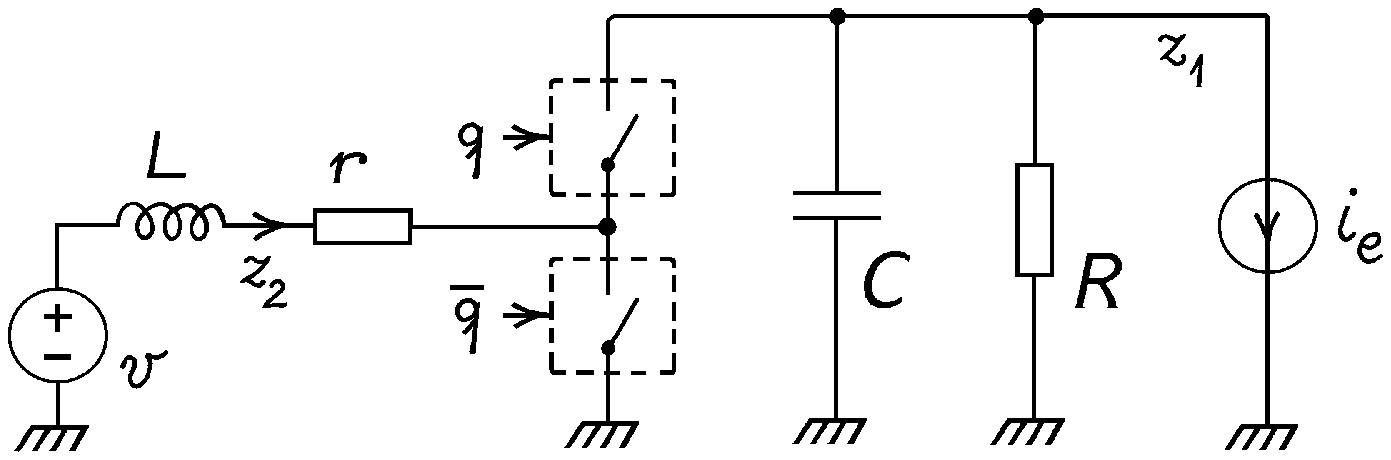} \vspace{-4mm}

{Figure $1$. A boost converter with input voltage $v$, input current
$z_2$, output voltage $y=z_1$ and disturbance current $i_e$. The
electronic switches are controlled by the binary signals $q$ and
$\bar q=1-q$, generated by a controller. \vspace{-4mm}}
\end{center}

\section{Problem setting and background} \label{sec2}
\setcounter{equation}{0} 

\ \ \ In this section, following \cite{Isi:95}, we introduce the
nonlinear plant and exosystem in Subsection 2.1 and define the error
feedback regulator problem for them in Subsection 2.2. We present the
definition of autonomous systems and of their local immersions in
Subsection 2.3 and then restate a well-known result from \cite{Isi:95}
which characterizes the solvability of the regulator problem, but with
added details about the smoothness of certain maps. Finally,
considering a linearization of the plant and the exosystem, we present
a simple proposition in Subsection 2.4 which follows immediately from
the results in \cite{Dav:76,HaPo:00}.

{\large\bf 2.1 The plant and the exosystem.} \m
We consider a nonlinear {\em plant} \vspace{-2mm}
\BEQ{nlplant}
  \dot x \m=\m f(x,u,w) \m, \qquad y \m=\m g(x,u,w) \m, \vspace{-2mm}
\end{equation}
with state $x(t)\in X\subset\rline^n$, control input $u(t)\in U\subset
\rline$ and output $y(t)\in Y\subset\rline$, where $X$, $U$ and $Y$
are open sets that contain the origin of the appropriate spaces. The
exogenous disturbance signal $w$ in \rfb{nlplant} is the state of the
nonlinear {\em exosystem} \vspace{-2mm}
\BEQ{nlexo}
  \dot w \m=\m s(w) \m, \vspace{-2mm}
\end{equation}
with $w(t)\in W\subset\rline^p$, where $W$ is an open set containing
the origin. The functions $f:X\times U \times W\to\rline^n$, $g:X
\times U\times W\to Y$ and $s:W\to W$ are assumed to be of class $C^2$
with $f(0,0,0)=0$, $g(0,0,0)=0$ and $s(0)=0$.
The {\em reference signal} is \vspace{-2.5mm}
\BEQ{Rajapaksa}
  y_r \m=\m q(w) \m. \vspace{-2mm}
\end{equation}
The function $q:W\to Y$ is assumed to be of class $C^2$ with $q(0)=0$.
The {\em tracking error} is \vspace{-2mm}
\BEQ{error}
  e(t) \m=\m y(t)-y_r(t) \m=\m g(x,u,w)-q(w) \m=\m h(x,u,w) \m.
\end{equation}
Clearly $h:X\times U\times W \to Y_e$ is a $C^2$ map satisfying
$h(0,0,0)=0$. Here $Y_e\subset\rline$ is open and $Y\subset Y_e$.
While we consider only SISO plants, the controller design procedure
described in this work also applies to MIMO plants. However, it is
unclear if the resulting controller will be of minimal order. This
is a topic for future research.

For the dynamics of the exosystem, we assume that the origin is a
Lyapunov stable equilibrium point and also that every point $w_0\in W$
is nonwandering. If $\phi$ denotes the flow generated by $s$ on $W$,
then a point $w_0\in W$ is called {\em nonwandering} if its trajectory
$\phi_t(w_0)$ is defined for all $t>0$ and for every neighborhood
$\Oscr$ of $w_0$ and every $T>0$, there exists $\tau>T$ such that
$\phi_{\tau}\m(\Oscr)\cap\Oscr$ is not empty (see \cite [Definition
5.2.2, page 236]{GuHo:83}). Equivalently, for each $w_0\in W$ there
exist sequences $(t_n)_{n=1}^\infty$ in $(0,\infty)$ and $(w_n)_{n=1}
^\infty$ in $W$ such that $\phi_{t_n}(w_n)$ exists and \vspace{-2mm}
\BEQ{Simy_teeth_operation}
   t_n\rarrow\infty\m,\qquad w_n\rarrow w_0\qquad \mbox{and} \qquad
   \phi_{t_n}(w_n)\rarrow w_0 \m. \vspace{-2mm}
\end{equation}
Denoting $\tilde t_n=-t_n$, $\tilde w_n=\phi_{t_n}(w_n)$, it follows
that $\phi_{\tilde t_n}(\tilde w_n)=w_n$ exists and \vspace{-2mm}
\BEQ{Bolt_won_200m}
   \tilde t_n\rarrow-\infty,\qquad \tilde w_n\rarrow w_0\qquad
   \mbox{and} \qquad \phi_{\tilde t_n}(\tilde w_n)\rarrow w_0 \m,
   \vspace{-2mm}
\end{equation}
which is exactly the same property as \rfb{Simy_teeth_operation}
but backwards in time. We mention that \cite{Isi:95,IsBy:90} make a
slightly stronger assumption about the exosystem: every point $w_0
\in W$ is assumed to be {\em Poisson stable}, which is like
\rfb{Simy_teeth_operation} with $w_n=w_0$, together with
\rfb{Bolt_won_200m} with $\tilde w_n=w_0$.

We define the real matrices \vspace{-2mm}
$$ A \m=\m \left[\frac{\partial f}{\partial x}\right]_{(0,0,0)},\ \
   B \m=\m \left[\frac{\partial f}{\partial u}\right]_{(0,0,0)},\ \
   P \m=\m \left[\frac{\partial f}{\partial w}\right]_{(0,0,0)},$$
$$ C \m=\m \left[\frac{\partial h}{\partial x}\right]_{(0,0,0)},\ \
   D \m=\m \left[\frac{\partial h}{\partial u}\right]_{(0,0,0)},\ \
   Q \m=\m \left[\frac{\partial h}{\partial w}\right]_{(0,0,0)}, \ \
   S \m=\m \left[\frac{\partial s}{\partial w}\right]_{(0)}.$$
The assumptions on $s$ imply that $S$ has all its eigenvalues on the
imaginary axis (see Remark 8.1.1 in \cite{Isi:95} for a closely
related statement, with practically the same proof). If $s$ is linear,
then the assumptions on $s$ also imply that $S$, when written in the
Jordan normal form, has no Jordan blocks of order larger than 1.
However, for a nonlinear $s$, $S$ can have nontrivial Jordan blocks;
for instance, see the exosystem in Example \ref{exmp1}. Our main
result is stated under the following assumption, which can be relaxed
(this is discussed in Remark \ref{relax}). For any square matrix $M$,
let $\sigma_p(M)$ denote the set of eigenvalues of $M$ and let $\rho(M)
=\{a\in\cline\m\big|\m a\notin\sigma_p(M)\}$.

\begin{assumption} \label{stab}
The matrix $A$ is stable (i.e. $\Re\l<0$ for all $\l\in\sigma_p(A)$).
\end{assumption}

The above assumption is natural given the fact that this work, as
discussed in the Introduction, builds on the results in
\cite{Dav:76,HaPo:00,ReWe:03} (see Proposition \ref{ReWe}) which have
been developed for exponentially stable plants. We will need the
following linearization at $(x,u,w)=(0,0,0)$ of the plant
\rfb{nlplant} and the error \rfb{error}: \vspace{-2mm}
\BEQ{plant} \m\ \ \ \ \ \dot x \m=\m Ax+Bu+Pw \m, \vspace{-2mm}
\end{equation}
\BEQ{linerror}
   \m\ \ \ \ \ e_l \m=\m Cx+Du+Qw \m.
\end{equation}
The transfer function of the linear system \rfb{plant}, \rfb{linerror}
from the control input $u$ to the linearized error $e_l$ is $\GGG(z)=
C(zI-A)^{-1}B+D$, which is defined for each $z\in\rho(A)$.

\begin{remark} \label{CDC}
This work generalizes the results in \cite{NaWe:14,WeNa:14} by relaxing
the assumption in those papers that the pair $\left(\sbm{C & Q},\sbm{A
& P\\ 0 &S} \right)$ be detectable. This higher generality makes the
main result and its proof much more complicated. In particular, we have
to consider local immersions of autonomous systems (see Definitions
\ref{autonomous} and \ref{immersion}) and derive new results for them
using the center manifold theorem (see Proposition \ref{Haragus}).
Furthermore, the exosystem in \cite{NaWe:14,WeNa:14} was linear with
only simple eigenvalues, but these conditions are dropped in the
present work, at the expense of new, more complex arguments. Letting
$S$ to contain non-trivial Jordan blocks allowed us to explore an
interesting and, to the best of our knowledge, novel structure for the
exosystem in Example \ref{exmp1}. Finally, two other examples from
\cite{NaWe:14,WeNa:14} have been further refined and expanded in this
work. \end{remark}

{\large \bf 2.2 The error feedback regulator problem.} \m
Consider the controller \vspace{-2.5mm}
\BEQ{gencon}
  \m\ \ \dot\xi \m=\m \eta(\xi,e) \m, \qquad u \m=\m \theta(\xi) \m,
  \vspace{-1mm}
\end{equation}
with state $\xi(t)\in X_c\subset\rline^{n_c}$, input $e$ and output
$u$, where $X_c$ is open and contains $0$. The functions $\eta:X_c
\times Y_e\to \rline^{n_c}$ and $\theta:X_c\to U$ are $C^2$ maps
satisfying $\m \eta(0,0)=0$ and $\theta(0)=0$. The closed-loop system
consisting of the plant \rfb{nlplant}, the exosystem \rfb{nlexo} and
the controller \rfb{gencon} is shown in Figure 2. We define the
matrices \vspace{-2mm}
\BEQ{shooting_in_Brussels}
   F \m=\m \left[\frac{\partial\eta}{\partial \xi}\right]_{(0,0)},\ \
   G \m=\m \left[\frac{\partial\eta}{\partial e}\right]_{(0,0)},\ \
   K \m=\m \left[\frac{\partial\theta}{\partial \xi}\right]_{(0)} \m,
   \vspace{-2mm}
\end{equation}
which determine the linearization of the controller. The {\em order}
of this controller is, by definition, $n_c$.

\begin{center} \vspace{-6mm}
\includegraphics[scale=0.53]{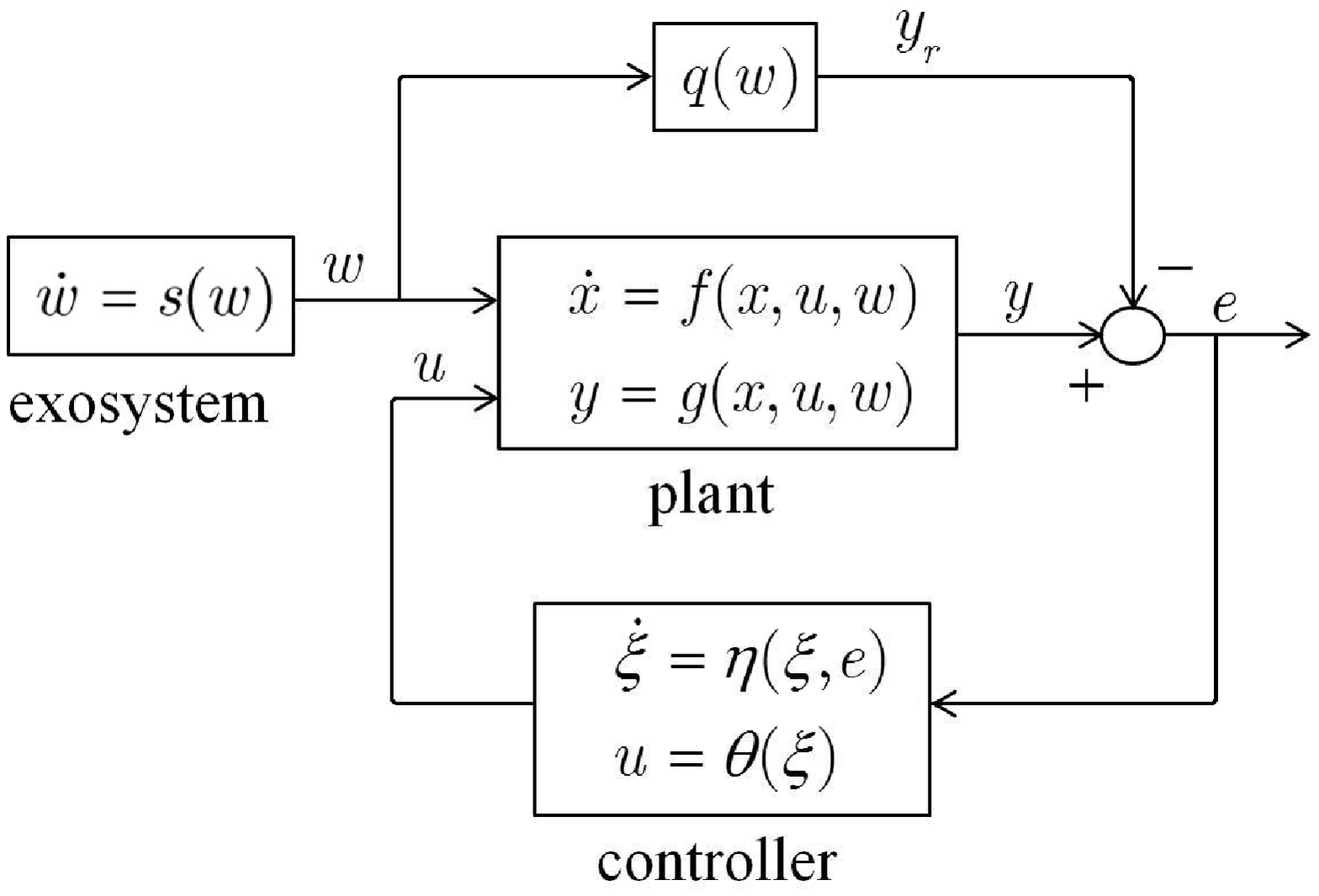} \vspace{-1mm}

{Figure 2. The closed-loop system of the plant and the controller,
driven\\ by the exosystem}
\end{center}

\begin{definition} \label{regprob}
The controller \rfb{gencon} is said to solve the {\em local error
feedback regulator problem} for the plant \rfb{nlplant}, the exosystem
\rfb{nlexo} and the error \rfb{error} if: \vspace{-2mm}
\begin{enumerate}
\item The equilibrium $(x,\xi)=(0,0)$ of the unforced closed-loop
system \vspace{-2mm}
$$ \dot x \m=\m f(x,\theta(\xi),0) \m, \qquad \dot\xi \m=\m
   \eta(\xi,h(x,\theta(\xi),0)) \m, \vspace{-2mm}$$
is locally exponentially stable. \vspace{-1mm}
\item The forced closed-loop system \vspace{-2mm}
\BEQ{Hanna_reg_to_Technion}
   \dot x \m=\m f(x,\theta(\xi),w) \m,\qquad \dot w \m=\m s(w) \m,
   \qquad \dot\xi \m=\m \eta(\xi,h(x,\theta(\xi),w)) \m, \vspace{-2mm}
\end{equation}
is such that for each initial condition $(x(0),\xi(0),w(0))$ in some
fixed neighborhood of $(0,0,0)$ in $X\times X_c\times W$, it has a
unique global (in time $t\geq 0$) solution $(x,\xi,w)$ and the
corresponding error $e$ defined in \rfb{error} satisfies \vspace{-2mm}
$$  \lim_{t\to+\infty}e(t) \m=\m \lim_{t\to+\infty} h(x(t),\theta
  (\xi(t)),w(t)) \m=\m 0 \m. $$
\end{enumerate}
\end{definition}


{\large \bf 2.3 The regulator equations.} \m
Lemma \ref{lm:IsBy} and Theorem \ref{IsBy}, presented below, are
well-known results in the nonlinear regulator theory. They were
first established in \cite{IsBy:90} under certain detectability
assumptions on the linearization of the plant, the error and the
exosystem. A more general version of Theorem \ref{IsBy} (along with
Definition \ref{immersion}) was then presented in \cite{Isi:95}.
Lemma \ref{lm:IsBy} and Theorem \ref{IsBy} do not require that $A$
be stable. Also, these results are derived in \cite{Isi:95} for
multi-input multi-output plants. For some comments on these
results, see Remark \ref{Putin_occupies_Crimea}.

\begin{lemma} \label{lm:IsBy}
Suppose that for the controller \rfb{gencon} the first condition in
Definition {\rm\ref{regprob}} (local exponential stability) holds.
This controller solves the local error feed\-back regulator problem
for the plant \rfb{nlplant}, the exosystem \rfb{nlexo} and the error
\rfb{error} if and only if there exist an open set $W^o\subset W$
containing zero and $C^2$ maps $\pi:W^o\to X$ and $\sigma:W^o\to
X_c$, with $\pi(0)=0$ and $\sigma(0)=0$, satisfying for every $w\in
W^o$ \vspace{-1mm}
\begin{align}
  &\frac{\partial\pi}{\partial w}s(w) \m=\m f(\pi(w),\theta(\sigma
  (w)),w) \m, \label{lm1}\\
  &\frac{\partial\sigma}{\partial w}s(w) \m=\m \eta(\sigma(w),0) \m,
  \label{lm2}\\
  & h(\pi(w),\theta(\sigma(w)),w) \m=\m 0 \m. \label{lm3}
\end{align}

\vspace{-2mm}
Moreover, in this case there exist open sets $Z\subset X\times X_c$
and $W^{oo}\subset W^o$, both containing zero, such that for any
$(x_0,\xi_0)\in Z$ and any $w_0\in W^{oo}$, the closed-loop system
\rfb{Hanna_reg_to_Technion} has a unique global (in time $t\geq 0$)
solution $(x,\xi,w)$ with $x(0)=x_0$, $\xi(0)=\xi_0$, $w(0)=w_0$ and
\BEQ{lemma_limit}
  \lim_{t\rarrow\infty} \|x(t)-\pi(w(t))\| \m=\m 0 \m,\qquad
  \lim_{t\rarrow\infty} \|\xi(t)-\sigma(w(t))\| \m=\m 0 \m.
\end{equation}
\end{lemma}

\begin{remark}
The statement of the above lemma in \cite{Isi:95} does not describe
the asymptotic response of the plant and the controller states
expressed by the limits in \rfb{lemma_limit}. These limits can be
obtained using the center manifold theory \cite{Car:81} which
guarantees the existence of a locally attractive invariant manifold
$\{(\pi(w),\sigma(w),w)|$ $w\in W^{oo}\}\subset Z\times W^{oo}$ for
the dynamics of the closed-loop system \rfb{Hanna_reg_to_Technion}.
\end{remark}

\begin{definition} \label{autonomous}
An {\em autonomous system} is a triple $(W_1,s_1,\g_1)$, where $W_1
\subset\rline^{r_1}$ is an open set containing the origin and the maps
$s_1:W_1\to\rline^{r_1}$ and $\g_1:W_1\to\rline$ are both of class
$C^2$ and they satisfy $s_1(0)=0$ and $\g_1(0)=0$. The {\em order} of
this system is $r_1$. The state trajectories $w_1$ and the
corresponding output functions $u_1$ of $(W_1,s_1,\g_1)$ are the
solutions (on any time interval) of the equations \vspace{-2mm}
\BEQ{aut_sys}
   \dot w_1 \m=\m s_1(w_1) \m, \qquad u_1 \m=\m \g_1(w_1)
   \vspace{-2mm}
\end{equation}
with $w_1(t)\in W_1$ and $u_1(t)\in\rline$. \vspace{-1.5mm}
\end{definition}

Note that the triple $(W,s,q)$ from \rfb{nlexo} and \rfb{Rajapaksa}
is an autonomous system.

\begin{definition} \label{immersion}
An autonomous system $(W_1,s_1,\g_1)$ is said to be {\em locally
immersed} into another autonomous system $(W_2,s_2,\g_2)$ if there
exists an open set $\widetilde W_1\subset W_1$ containing the origin
and a $C^2$ map $\tau:\widetilde W_1\to W_2$ such that $\tau(0)=0$
and \vspace{-2mm}
$$ \frac{\partial\tau}{\partial w_1}s_1(w_1) \m=\m s_2(\tau(w_1)) \m,
   \qquad \g_1(w_1) \m=\m \g_2(\tau(w_1)) \FORALL w_1\in
   \widetilde W_1 \m. \vspace{-2mm} $$
\end{definition}

For systems as above, all the output functions of \m $(\widetilde W_1,
s_1,\g_1)$ are also output functions of \m $(W_2,s_2,\g_2)$. However,
the eigenvalues of the linearization of $s_1$ need not be eigenvalues
of the linearization of $s_2$, see Example \ref{exmp2}.

\begin{theorem} \label{IsBy} {\bf (Isidori and Byrnes)}
There exists a controller of the form \rfb{gencon} that solves the
local error feedback regulator problem for the plant \rfb{nlplant},
the exosystem \rfb{nlexo} and the error \rfb{error} if and only if
there exist an open set $W^o\subset W$ with $0\in W^o$ and $C^2$
maps $\pi:W^o\to X$ and $\g:W^o\to U$, with $\pi(0)=0$ and $\g(0)=0$,
\vspace{-2mm} such that:
\begin{enumerate}
\item For each $w\in W^o$, $\pi$ and $\gamma$ satisfy the nonlinear
{\em regulator equations} \vspace{-1mm}
\begin{align}
&\frac{\partial\pi}{\partial w}s(w) \m=\m f(\pi(w),\g(w),w) \m,
\label{nlreg1}\\[1mm] & h(\pi(w),\g(w),w) \m=\m 0 \m. \label{nlreg2}
\end{align}
\vspace{-9mm}
\item The autonomous system $(W^o,s,\g)$ is locally immersed into
$(E^o,\phi,\l)$, and the linearization of the latter system, $\Phi=
\left[\frac{\partial\phi}{\partial\zeta}\right]_{(0)}$ and $\L=
\left[\frac{\partial\l}{\partial\zeta}\right]_{(0)}$, is such that
for some column vector $N$ (of suitable dimension),  \vspace{-1mm}
\BEQ{imm_stab_det}
   \left(\bbm{A & B\L\\ NC & \Phi+N D\L},\bbm{B \\ 0},\bbm{C & D\L}
   \right) \mbox{is stabilizable and detectable} \m.
\end{equation}
\end{enumerate}
\end{theorem}
\vspace{-2mm}

The intuitive meaning of the system \rfb{imm_stab_det} can be
understood from Proposition \ref{ReWe} below: it is the linearization
of the feedback interconnection of the plant \rfb{nlplant} with the
internal model $(E^o,\phi,\l)$, the latter receiving its input $e$ via
the matrix $N$.

Suppose that the conditions in Theorem \ref{IsBy} hold. Let the
matrices $K$, $L$ and $M$ be such that
$$\bbm{A & B\L & BM \\ NC & \Phi+N D\L & 0\\ L C & LD\L & K}$$
is exponentially stable. Then the controller \rfb{gencon} with
$$ \xi\m=\m \bbm{\xi_0&\xi_1}^\top\m, \qquad \eta(\xi,e) \m=\m
   \bbm{K\xi_0+Le\\ \phi(\xi_1)+Ne}\m, \qquad \theta(\xi) \m=\m
   M\xi_0+\l(\xi_1)\m,$$
solves the local error feedback regulator problem for the plant
\rfb{nlplant}, the exosystem \rfb{nlexo} and the error \rfb{error}
\cite{Isi:95}. A possible choice for $K$, $L$ and $M$, suggested by
the observer-based controller design approach, is to choose $L$ and
$M$ such that
$$ \bbm{A & B\L\\ NC & \Phi+N D\L}+\bbm{B\\0}M \qquad \textrm{and}
   \qquad \bbm{A & B\L\\ NC & \Phi+N D\L}-L\bbm{C & D\L}$$
are exponentially stable and then choose \vspace{-1mm}
$$ K \m=\m \bbm{A & B\L\\ NC & \Phi+N D\L} + \bbm{B\\0}M - L
   \bbm{C & D\L}. \vspace{-1mm}$$
In this design, the order of the controller (the dimension of $\xi$)
is $n_c=n+2\nu$, where $\nu$ is the order of $(E^o,\phi,\l)$.

\begin{remark} \label{Putin_occupies_Crimea}
The function $g$ in \rfb{nlplant} is considered in \cite{Isi:95} to be
independent of $u$ and thus $D=0$ in that work, which makes the
formulas simpler. However, the proof remains almost the same if we
include the matrix $D$. In \cite{Isi:95} there is a different
assessment of the smoothness of $\pi$ and $\gamma$ ($C^1$ in place of
$C^2$). The $C^2$ smoothness can be established by a similar argument.
\end{remark}

{\large \bf 2.4 Solution to the linear regulator problem.} \m
The following proposition is based on the results in \cite{Dav:76}
which address the linear error feedback regulator problem. Theorem 6
in \cite{HaPo:00} (and also \cite{ReWe:03}) extends the results in
\cite{Dav:76} to a class of infinite-dimensional systems. To simplify
our presentation, we state this proposition under the assumption that
$0\in\sigma_p(S)$, but the proposition remains true even when this
assumption is dropped. Recall the notation $\GGG$ introduced after
\rfb{linerror}.

\begin{proposition} \label{ReWe}
Suppose that $A$ is stable and $S$ has eigenvalues $\{\alpha_0,\pm i
\alpha_1,\ldots\pm i\alpha_q\}$, with $\alpha_0=0$ and each $\alpha_j
\in\rline$, such that the geometric multiplicity of every eigenvalue
is 1 and $\GGG(i\alpha_j)\neq 0$ for all $j\in\{0,1,\ldots q\}$. Let
the vector $C_c^\top\in\rline^{p}$ be such that the pair $(C_c,S)$ is
detectable (equivalently, observable). Then there exists a vector
$B_c\in\rline^p$ such that the linear controller \vspace{-2mm}

\BEQ{lincontr}
   \dot z_c \m=\m Sz_c + B_c e_l \m, \qquad u \m=\m C_c z_c \m,
   \vspace{-2mm}
\end{equation}
for the linearized plant \rfb{plant}, \rfb{linerror} renders the
closed-loop system stable, i.e. \vspace{-2mm}
\BEQ{Ascr}
  \Ascr \m=\m \bbm{A & B C_c\\ B_c C & S+B_c D C_c} \vspace{-1.5mm}
\end{equation}
is a stable matrix. For any such $B_c$, $C_c$, the controller in
\rfb{lincontr} solves the linear error feedback regulator problem
for the linearized plant \rfb{plant}, the linearization of the
exosystem \rfb{nlexo} and the linearized error \rfb{linerror},
and moreover in this case the linearized error $e_l(t)$ converges
to zero exponentially.
\end{proposition}

On the basis of the results in \cite{HaPo:00}, the next remark
states a concrete way of choosing $B_c$ in the above proposition
so that $\Ascr$ in \rfb{Ascr} is stable.

\begin{remark} \label{clarify}
For each $j\in\{0,1,\ldots q\}$, denote by $m_j$ the algebraic
multiplicity of $i\alpha_j\in\sigma_p(S)$. Then clearly $p=m_0
+2\sum_{j=1}^q m_j$. For each $j\in\{0,1,\ldots q\}$, choose the
set $\{a_{j1},a_{j2},\ldots a_{jm_j}\}\subset\cline$ such that
all the zeros of the polynomial $z^{m_j}+\GGG(i\alpha_j)\sum_
{k=1}^{m_j} a_{jk}z^{m_j-k}$ lie in the open left half-plane.
Since $\GGG(0)\in\rline$, we can choose $\{a_{01},a_{02},\ldots
a_{0m_0}\}\subset\rline$. According to \cite[Theorem 6]{HaPo:00},
if we choose $B_c$ such that the transfer function $\CCC$ of the
linear controller \rfb{lincontr} is given by \vspace{-2mm}
\BEQ{TF_HaPo}
 -\m \CCC(z) \m=\m \sum_{j=0}^q\sum_{k=1}^{m_j} \frac{a_{jk}\m\e^k}
 {(z-i\alpha_j)^k} + \sum_{j=1}^q\sum_{k=1}^{m_j}\frac{\bar a_{jk}
 \m\e^k}{(z+i\alpha_j)^k} \m, \vspace{-2mm}
\end{equation}
where $\e>0$ is a constant and a {\em bar} indicates a complex
conjugate, then for all $\e$ sufficiently small $\Ascr$ in
\rfb{Ascr} is stable.

Given $S$ and $C_c$, we now describe how $B_c$ can be chosen so
that the transfer function of \rfb{lincontr} is $\CCC$. Let $J$ be
an upper-triangular Jordan normal form for $S$, with the set of
Jordan blocks on its diagonal being $\{J_0,J_1,J_{-1},\ldots
J_{q},J_{-q}\}$ (in this order, $J_j$ and $J_{-j}$ correspond to
the eigenvalues $i\alpha_j$ and $-i \alpha_j$, respectively). Let
\vspace{-1mm}
$$ T \m=\m \bbm{V_0\ V_1\ V_{-1}\ \ldots V_{q}\ V_{-q}}\in
   \cline^{p\times p} \vspace{-1mm} $$
be the matrix of generalized eigenvectors of $S$ such that $J=
T^{-1}ST$, $V_0\in\rline^{p\times m_0}$ and $V_{-j}\in\cline^{p
\times m_j}$ is the element-wise complex conjugate of $V_j\in
\cline^{p\times m_j}$ for all $j\in\{1,2,\ldots q\}$. Define
$\Cscr_c=C_cT$. Then \vspace{-1mm}
$$ \Cscr_c \m=\m \bbm{\Cscr_0\ \Cscr_1\ \Cscr_{-1} \ldots \Cscr_q
   \ \Cscr_{-q}} \in \cline^{1\times p}\m, \vspace{-1mm} $$
where $\Cscr_0\in\rline^{1\times m_0}$ and $\Cscr_j\in\cline^{1
\times m_j}$ and $\Cscr_{-j}$ is the element-wise complex conjugate
of $\Cscr_j$ for each $j\in\{1,2,\ldots q\}$. We will construct
$\Bscr_c\in\cline^{p\times1}$ with the structure \vspace{-1mm}
$$ \Bscr_c^\top \m=\m \bbm{\Bscr_0^\top\ \Bscr_1^\top\ \Bscr_{-1}
   ^\top\ldots \Bscr_q^\top \ \Bscr_{-q}^\top} \m, \vspace{-1mm} $$
where $\Bscr_0\in\rline^{m_0\times 1}$ and $\Bscr_j\in\cline^{m_j
\times1}$ and $\Bscr_{-j}$ is the element-wise complex conjugate
of $\Bscr_j$ for each $j\in\{1,2,\ldots q\}$, such that
\vspace{-2mm}
\BEQ{TF_subsystem}
  \Cscr_j(zI-J_j)^{-1}\Bscr_j \m=\m \sum_{k=1}^{m_j} \frac{-a_{jk}
  \m\e^k}{(z-i\alpha_j)^k} \FORALL j\in\{0,1,\ldots q\}\m.
  \vspace{-2mm}
\end{equation}
Using \rfb{TF_HaPo} and \rfb{TF_subsystem} it follows that $\Cscr_c
(zI-J)^{-1}\Bscr_c=\CCC(z)$ for all $z\notin\sigma_p(S)$. If we let
$B_c=T\Bscr_c$, it is then easy to verify that $B_c\in\rline^{p
\times1}$ and $C_c(zI-S)^{-1}B_c=\CCC(z)$ for all $z\notin\sigma_p
(S)$, i.e. for this choice of $B_c$ the matrix $\Ascr$ is stable. It
now only remains to choose $\Bscr_j$ such that \rfb{TF_subsystem}
holds. This is done below.

Fix $j\in\{0,1,\ldots q\}$ and $\e>0$ sufficiently small. Let
$\Bscr_j^\top = \bbm{b_{j1}\ b_{j2}\ \ldots b_{jm_j}}$ and $\Cscr_j=
\bbm{c_{j1}\ c_{j2}\ \ldots c_{jm_j}}$. Solve the following
algebraic equations for $\{b_{j1},b_{j2},\ldots b_{jm_j}\}$:
\vspace{-2mm}
\BEQ{eps_equation}
 \sum_{l=1}^{m_j-k+1} c_{jl}b_{j(l+k-1)} \m=\m -a_{jk}\e^k
\FORALL k\in\{1,2,\ldots m_j\} \m. \vspace{-1mm}
\end{equation}
This set of equations has a unique solution. Indeed, since
$(C_c,S)$ is a detectable pair, so is $(\Cscr_c,J)$. Hence
$(\Cscr_j,J_j)$ is also a detectable pair, which implies that
$c_{j1}\neq0$. The solution to \rfb{eps_equation} can be computed
sequentially as follows: $b_{jm_j}=-\e^{m_j}a_{jm_j}/c_{j1}$ and
given $\{b_{j(k+1)},b_{j(k+2)},\ldots b_{jm_j}\}$ for any $1\leq
k\leq m_j-1$, \vspace{-2mm}
$$ b_{jk} \m=\m -\Big(a_{jk}\e^k + \sum_{l=2}^{m_j-k+1} c_{jl}
   b_{j(l+k-1)}\Big)\Big/c_{j1} \m. \vspace{-2mm} $$
It is straightforward to verify that for $\Bscr_j$ chosen such that
\rfb{eps_equation} holds, \rfb{TF_subsystem} holds.
\end{remark}

\begin{remark}
Let us consider for the moment the frequently encountered situation
when $m_j=1$ for all $j\in\{0,1,\ldots q\}$ (i.e., $S$ can be
diagonalized). Then the method described in Remark \ref{clarify}
becomes much simpler. We write $a_j$ instead of $a_{j1}$. The numbers
$a_j$ must be chosen such that $\Re[a_j\GGG(i\alpha_j)]>0$. The
transfer function $\CCC$ of the linear controller from \rfb{lincontr}
should be given by \vspace{-3mm}
$$ -\m\CCC(z) \m=\m \frac{\e a_0}{z} + \e\sum_{j=1}^q \left(
   \frac{a_j}{z-i\alpha_j} + \frac{\bar a_j}{z+i\alpha_j}\right) ,
   \vspace{-2mm}$$
with $\e>0$ sufficiently small, and the vector $B_c$ must be chosen
accordingly. Under a passivity assumption on the plant, any $\e>0$
will work, see \cite{ReWe:03}.
\end{remark}

\begin{remark} \label{alternate}
The transfer function of the controller in \rfb{lincontr} has poles at
$\pm i\alpha_j$ ($j\in\{0,1,\ldots p\}$). The necessity of the
condition $\GGG(i\alpha_j)\neq 0$ for the solvability of the local
error feedback regulator problem can be seen also from the following
fact: in a stable closed-loop system there cannot be an unstable
pole-zero cancelation in the product of the plant and controller
transfer functions. This is well-known but not easy to find in the
literature, see for instance \cite{AndGev}.
\end{remark} \vspace{-8mm}

\section{Reduced order local immersion} \label{sec3}
\setcounter{equation}{0} \vspace{-2mm} 

\ \ \ We start this section by stating (in Theorem \ref{CMT}) a
version of the center manifold theorem based on \cite[Theorem 2.9 and
Corollary 2.12]{HaIo:11}. We remark that several hypothesis in
\cite{HaIo:11}, which considers infinite-dimensional plants, hold
trivially for finite-dimensional plants (see \cite[Exercise
2.8]{HaIo:11}). Afterwards we show (in Proposition \ref{Haragus}) that
for certain autonomous systems $(W_1,s_1,\g_1)$ a given local
immersion $(W_2,s_2, \g_2)$ can be reduced, by eliminating all the
eigenvalues of the linearization of $s_2$ around the origin that are
not on the imaginary axis, to obtain a local immersion of a lower
order. The proof of this is based on the center manifold theorem.

We denote by $I$ and by $\Bscr(w,d)$ the identity operator and the
open ball of radius $d>0$ around $w$, respectively, on any space (the
space will be clear from the context). We use the concept of {\em
spectral projection operator}, see \cite[p.~32]{HaIo:11} for
details.

\vspace{-1mm} \begin{theorem} \label{CMT}
Consider an autonomous system $(W_2,s_2,\g_2)$ of order $r_2$ and
define $S_2=\left[\frac{\partial s_2}{\partial\zeta}\right]_
{(0)}$. Let $\PPP:\rline^{r_2}\to\rline^{r_2}$ denote the spectral
projection operator corresponding to the set of all eigenvalues of
$S_2$ that lie on the imaginary axis. Let $E_0$ and $E_h$ be the
images of $\m\PPP$ and $I-\PPP$, respectively. Then there exists a
$C^2$ map $\Psi:E_0\to E_h$ with $\Psi(0)=0$, $\left[\frac{\partial
\Psi}{\partial\zeta}\right]_{(0)}=0$ and a ball $\Bscr(0,d)\subset
W_2$ such that the {\em center manifold} \vspace{-4mm}
$$ \m\hspace{20mm} \Mscr_0 \m=\m \{w_0+\Psi(w_0)\m\big|\m w_0\in
   E_0\} \subset\m \rline^{r_2} \vspace{-3mm}$$
has the following properties: \vspace{-4mm}
\begin{enumerate}
 \item $\Mscr_0$ is locally invariant, which means the following:
 for each state trajectory $w_2$ of $(W_2,s_2,\g_2)$ satisfying
 $w_2(0)\in\Mscr_0\cap\Bscr(0,d)$ and $w_2(t)\in\Bscr(0,d)$ for all
 $t\in[0,T]$, we have $w_2(t)\in\Mscr_0$ for all $t\in[0,T]$.
 \vspace{-2mm}
 \item Every state trajectory $w_2$ of $(W_2,s_2,\g_2)$
 satisfying $w_2(t)\in\Bscr(0,d)$ for all $t\in(-\infty,\infty)$
 also satisfies $w_2(0)\in\Mscr_0$.
\end{enumerate} \vspace{-4mm}

Moreover, for each state trajectory $w_2$ of $(W_2,s_2,\g_2)$
satisfying $w_2(t)\in\Mscr_0\cap\Bscr(0,d)$ for all $t\in[0,T)$,
the function $w(t)=\PPP w_2(t)$ satisfies \vspace{-3.5mm}
\BEQ{reduced_sys}
  w_2 \m=\m w +\Psi(w)\m, \qquad \frac{\dd w}{\dd t} \m=\m
  \PPP s_2(w+\Psi(w)) \FORALL t\in[0,T) \m.
\end{equation}
\end{theorem}

\begin{proposition} \label{Haragus}
Suppose that for the dynamics \rfb{aut_sys} of the autonomous system
$(W_1,s_1,\g_1)$, the origin is a Lyapunov stable equilibrium point
and every $w_0\in W_1$ is nonwandering (as defined before
\rfb{Simy_teeth_operation}). Let $(W_1,s_1,\g_1)$ be locally immersed
into an autonomous system $(W_2,s_2,\g_2)$ of order $r_2$ via the map
$\tau$ (as in Definition {\rm\ref{immersion}}). Denote
$S_2=\left[\frac{\partial s_2}{\partial \zeta}\right]_{(0)}$ and $\G_2=\left[\frac{\partial \g_2}{\partial \zeta}\right]_{(0)}$.

Then there exists an autonomous system $(W_3,s_3,\g_3)$ with linearization $S_3=\left[\frac{\partial s_3}{\partial \zeta}\right]_{(0)}$ and $\G_3=\left[\frac{\partial \g_3}{\partial \zeta}\right]_{(0)}$ such that $(W_1,s_1,\g_1)$ is locally immersed into $(W_3,s_3,\g_3)$, $S_3$ is isomorphic to the restriction of $S_2$ to one of its invariant subspaces and all the eigenvalues of $S_3$ are on the
imaginary axis.

Furthermore if the pair $(\G_2,S_2)$ is detectable, then so is the
pair $(\G_3,S_3)$. \vspace{-2mm}
\end{proposition}

\begin{proof}
Due to our Lyapunov stability assumption for \rfb{aut_sys}, there
exists an open set $W_0\subset W_1$ containing the origin and
$\delta>0$ such that $\overline{\Bscr(0,2\delta)}\subset W_1$ and
if $w_1(0)\in W_0$, then for the solution $w_1$ of \rfb{aut_sys},
$\|w_1(t)\|<\delta$ holds for all $t\geq 0$.

We claim that for all $w_1(0)\in W_0$, \rfb{aut_sys} has a
well-defined trajectory backwards in time and $\|w_1(t)\|<2\delta$ for
all $t<0$. To prove this, fix $w_1(0)\in W_0$. Then there exists a
maximal interval $(T,0]$ with $T<0$ on which \rfb{aut_sys} has a
unique solution $w_1$ and if $|T|<\infty$, then there exists a
$t_0\in(T,0)$ such that $\|w_1(t_0)\|>2\delta$. All this follows by
applying the standard results about existence and uniqueness of
solutions for ordinary differential equations (ODEs), evolving forward
in time, to the ODE $\dot w=-s_1(w)$ (see, e.g., \cite{Kha:02} or
\cite[Theorem II.2 and Corollary II.3]{JaWe:09}). Thus, to prove our
claim, we have to prove that $\|w_1(t_0)\|>2 \delta$ cannot occur for
any $t_0<0$.

Suppose that there exists a $t_0\in(T,0)$ with $\|w_1(t_0)\|>2\delta$.
Clearly, $w_1$ can be extended uniquely to a solution of \rfb{aut_sys}
on $(T,\infty)$ (recall that $w(t)\in W_1$ for all $t\in(T,\infty)$
and $w_1(0)\in W_0$). Denote by $\phi$ the flow generated by $s_1$ on
$W_1$. Since every point in $W_1$ is nonwandering, according to
\rfb{Simy_teeth_operation} we can find a sequence $(w_n)$ in $W_1$ and
a sequence $(t_n)$ in $(0,\infty)$ such that $t_n\rarrow \infty$,
$w_n\rarrow w_1(t_0)$ and $\phi_{t_n}(w_n)\rarrow w_1(t_0)$.
According to the continuous dependence of solutions of ODEs on the
initial state (\cite[Chapter 3]{Kha:02}) we have, in addition, that
$\phi_{-t_0}(w_n)\rarrow \phi_{-t_0}(w_1(t_0))=w_1(0)$. Thus, for $n$
sufficiently large, $\phi_{-t_0}(w_n)\in W_0$ and hence for $t>0$ the
trajectory $\phi_{t-t_0}(w_n)$ remains in $\Bscr(0,\delta)$. This
contradicts the earlier conclusion that the same trajectory has
$w_1(t_0)$ as one limit point. Thus we have proved our claim.

We complete this proof by applying the center manifold theorem to
$(W_2,s_2,\g_2)$. Recall the open ball $\Bscr(0,d) $, the subspaces
$E_0,E_h\subset\rline^{r_2}$, the center manifold $\Mscr_0$, the
nonlinear map $\Psi:E_0\to E_h$ and the projector $\PPP:\rline^{r_2}
\to E_0$ from Theorem \ref{CMT}. The system $(W_1,s_1,\g_1)$ is
locally immersed into $(W_2,s_2,\g_2)$ via the $C^2$ map $\tau$. First
we show that the image of any state trajectory of \rfb{aut_sys},
starting from some $w_1(0)\in W_0$, under $\tau$ is contained in
$\Mscr_0$. For this we suppose that $\delta>0$ is sufficiently
small such that $\tau\Bscr(0,2\delta) \subset\Bscr(0,d)$ ($\delta$ can
be made arbitrarily small by choosing $W_0$ appropriately). The claim
established earlier implies that for each $w_1(0)\in W_0$, there
exists a unique solution $w_1:\rline \rarrow W_1$ for \rfb{aut_sys}
and $\|w_1(t)\|<2\delta$ for all $t\in\rline$. This and $\tau\Bscr
(0,2\delta)\subset\Bscr(0,d)$ imply, via Theorem \ref{CMT}, that
\vspace{-1mm}
\BEQ{arkia_call}
  \tau w_1(t)\in\Mscr_0\cap\Bscr(0,d) \FORALL t\in\rline \m.
  \vspace{-2mm}
\end{equation}

Next we construct $(W_3,s_3,\g_3)$ with the desired properties. Denote
$r_3=\dim E_0$ and let $\Jscr$ be the linear isomorphism from $E_0$ to
$\rline^{r_3}$. Define $W_3=\Jscr(E_0\cap \Bscr(0,d))$. Clearly $W_3$
is an open set containing the origin. For any point $w_1(0)\in W_0$,
\rfb{arkia_call} and \rfb{reduced_sys} give that the function
$w_3=\Jscr\PPP w_2$, where $w_2$ is the state trajectory of
$(W_2,s_2,\g_2)$ with $w_2(0) =\tau w_1(0)$, is such that $w_3(t)\in
W_3$ for all $t\in[0,\infty)$ and on this interval \vspace{-2mm}
$$  \frac{\dd w_3}{\dd t} \m=\m \Jscr\PPP s_2(\Jscr^{-1}w_3 +
    \Psi(\Jscr^{-1}w_3)) \m.$$
This motivates to define the maps $s_3:W_3\to\rline^{r_3}$ and $\g_3:
W_3\to\rline$ by \vspace{-1mm}
$$  s_3(w) \m=\m \Jscr\PPP s_2(\Jscr^{-1}w+\Psi(\Jscr^{-1}w))\m,
  \qquad \g_3(w) \m=\m \g_2(\Jscr^{-1}w+\Psi(\Jscr^{-1}w))
  \m.\vspace{-1mm}$$
Clearly $(W_3,s_3,\g_3)$ is an autonomous system of order $r_3$. It
is easy to check, using the definitions of $\PPP$ and $\Psi$ and the
facts \vspace{-2mm}
$$ \tau(w_0)\in\Mscr_0\cap\Bscr(0,d) \m, \qquad\PPP\tau(w_0)+\Psi
   (\PPP\tau(w_0)) \m=\m \tau(w_0) \FORALL w_0\in W_0 \vspace{-1mm}$$
that $(W_1,s_1,\g_1)$ is locally immersed in $(W_3,s_3,\g_3)$ via the
map $\Jscr\PPP\tau$. The linearization of $s_3$ around the origin is
$S_3=\Jscr\PPP S_2\Jscr^{-1}$, i.e. $S_3$ is isomorphic to the
restriction of $S_2$ to its invariant subspace $E_0$ and $\sigma_p(S_3)
=\sigma_p(S_2)\cap i\rline$. This completes the proof of the main claim
of the proposition.

The linearization of $\g_3$ around the origin is $\G_3=\G_2\Jscr^{-1}$.
Assume that the pair $(\G_2,S_2)$ is detectable but the pair
$(\G_3,S_3)$ is not. Then, using the Hautus test (see, for instance,
\cite[p.~14]{KnIsFl:93}), it follows that for some $p\in\cline$
with $\Re p \geq0$, some $w\in\rline^{r_3}$ and denoting $v=\Jscr^{-1}w
\in E_0$, \vspace{-1mm}
\BEQ{modified_det}
  0 \m=\m \bbm{S_3 w- pw \\ \G_3 w} \m=\m \bbm{\Jscr\PPP S_2v-p\Jscr v
  \\ \G_2 v} \m=\m \bbm{\PPP S_2v-pv \\ \G_2 v} \m.\vspace{-1mm}
\end{equation}
To get the last equality we have used the fact that $\PPP S_2v-pv
\in E_0$ and $\Jscr:E_0\to\rline^{r_3}$ is an isomorphism. Since
$E_0$ is an invariant subspace for $S_2$, $v\in E_0$ and $\PPP$ is
the projection onto $E_0$, we get that $\PPP S_2v-pv=S_2v-pv$. This
and \rfb{modified_det} give the contradiction that $(\G_2,S_2)$
is not detectable. Thus if $(\G_2,S_2)$ is detectable, then $(\G_3,
S_3)$ must also be detectable.
\end{proof}

\begin{remark} \label{pp_verify}
Let $(W_1,s_1,\g_1)$, $(W_2,s_2,\g_2)$, $(W_3,s_3,\g_3)$, $S_2$ and
$S_3$ be as in the above proposition. In particular, $(W_2,s_2,\g_2)$
is a local immersion of order $r_2$ for $(W_1,s_1,\g_1)$ and $(W_3,
s_3,\g_3)$ is a local immersion of order $r_3$ for $(W_1,s_1,\g_1)$.
Since $S_3$ is isomorphic to the restriction of $S_2$ to one of its
invariant subspaces, it follows that $r_3\leq r_2$ and $\sigma_p(S_3)
\subset \sigma_p(S_2)$ and furthermore, since $\sigma_p(S_3) \subset
i\rline$, if $\sigma_p(S_2)\not\subset i\rline$, then $r_3<r_2$. Hence
if $(W_2,s_2,\g_2)$ is a local immersion of smallest order for
$(W_1,s_1,\g_1)$, then clearly $r_3<r_2$ cannot occur and so
$\sigma_p(S_2)\subset i\rline$.
\end{remark}

\vspace{-8mm}
\section{A minimal order controller} \label{sec4}
\setcounter{equation}{0} 
\vspace{-2mm}

\ \ \ Using the notation and assumptions from Subsection 2.1, in this
section we construct a low order controller that solves the local
error feedback regulator problem. Under some technical assumptions
this controller is of minimal order. This controller can be viewed as
a nonlinear version of the controller \rfb{lincontr}. Our main
theorem below is complicated, but if we are willing to compromise a
bit on the generality, then we get a much more simple version of it,
stated below as Corollary \ref{nlDcon}. The theorem states the
existence of the desired low order controller, but we also provide a
way to construct such a controller, see Remark \ref{contr_design}.

\noindent \vspace{-8mm}
\begin{framed} \vspace{-5mm}
\begin{theorem} \label{newmain}
Suppose that $A$ is stable and there exist $C^2$ maps $\pi:W^o\to X$
and $\g:W^o\to U$, where $W^o\subset W$ is an open set containing
zero, that satisfy the nonlinear regulator equations \rfb{nlreg1} and
\rfb{nlreg2}, with $\pi(0)=0$ and $\g(0)=0$. Let the autonomous system
$\{W^o,s,\g\}$ be locally immersed into the autonomous system $\{E^o,
\phi,\l\}$, where $E^o\subset\rline^\nu$ is an open set containing zero
and $\phi:E^o\to\rline^\nu$ and $\l:E^o\to\rline$ are $C^2$ maps
satisfying $\phi(0)=0$ and $\l(0)=0$. Assume that the linearizations
$\Phi=\left[\frac{\partial\phi}{\partial\zeta}\right]_{(0)}$ and $\L=
\left[\frac{\partial\l}{\partial\zeta}\right]_{(0)}$ satisfy the
following conditions: \vspace{-1.5mm}
\BEQ{key_found}
   (\L,\Phi) \mbox{ is detectable,} \qquad \GGG(p) \not= 0\ \ \ \
   \forall\ p\in\sigma_p(\Phi)\cap i\rline \m. \vspace{-1.5mm}
\end{equation}

If $\sigma_p(\Phi)\subset i\rline$, then there exists a $B_c\in
\rline^{\nu}$ such that the controller \vspace{-1.5mm}
\BEQ{contr}
  \dot\xi \m=\m \phi(\xi) + B_c e \m, \qquad u \m=\m \l(\xi) \m,
  \vspace{-1.5mm}
\end{equation}
solves the local error feedback regulator problem for the plant
\rfb{nlplant}, the exosystem \rfb{nlexo} and the error \rfb{error}.

If $\sigma_p(\Phi)\not\subset i\rline$, then there is a local
immersion $\{E^o_s,\phi_s,\l_s\}$ of smaller order for $\{W^o,s,\g\}$,
i.e. $E^o_s\subset\rline^{\nu_s}$ with $\nu_s<\nu$, such that denoting
$\Phi_s=\left[\frac{\partial\phi_s}{\partial\zeta}\right]_{(0)}$ and
$\L_s=\left[\frac{\partial\l_s}{\partial\zeta}\right]_{(0)}$, the pair
$(\L_s,\Phi_s)$ is detectable and $\sigma_p(\Phi_s)\subset\sigma_p
(\Phi)\cap i\rline$. In this case, there exists $B_{cs}\in\rline
^{\nu_s}$ such that the controller \vspace{-1.5mm}
\BEQ{contrs}
  \dot\xi \m=\m \phi_s(\xi) + B_{cs} e \m, \qquad u \m=\m \l_s(\xi)
  \m, \vspace{-1.5mm}
\end{equation}
solves the local error feedback regulator problem for the plant
\rfb{nlplant}, the exosystem \rfb{nlexo} and the error \rfb{error}.

Let $\nu>0$ be the smallest integer such that there exist some $\pi$
and $\g$ as above and a local immersion $\{E^o,\phi,\l\}$ of order
$\nu$ for $\{W^o,s,\g\}$, with linearizations $\Phi=\left[\frac
{\partial\phi}{\partial\zeta}\right]_{(0)}$ and $\L=\left[\frac
{\partial\l}{\partial\zeta}\right]_{(0)}$, for which the conditions
in \rfb{key_found} hold. Then $\sigma_p(\Phi)\subset i\rline$ and the
corresponding controller \rfb{contr} of order $\nu$ is minimal, i.e.
it is of the lowest possible order among controllers of the form
\rfb{gencon} that solve this error feedback regulator problem.
\end{theorem}
\vspace{-5mm} \end{framed} \vspace{-2mm}

\begin{remark} \label{relation_IsBy}
Theorem \ref{newmain} can be regarded as a refinement/extension of
Theorem \ref{IsBy} when the plant is locally exponentially stable.
Following Theorem \ref{IsBy}, the above theorem assumes that the
nonlinear regulator equations \rfb{nlreg1} and \rfb{nlreg2} are solved
by functions $\pi$ and $\g$ defined on $W^o$ and that $(W^o,s,\g)$ is
locally immersed into the internal model $(E^o,\phi, \l)$. But instead
of the condition \rfb{imm_stab_det} in Theorem \ref{IsBy} we assume in
Theorem \ref{newmain} that \rfb{key_found} holds, which is implied by
\rfb{imm_stab_det}. In the linear SISO case, the internal model can be
chosen to be a copy of the exosystem, but in the nonlinear SISO case
this is not always possible, see Example \ref{exmp2}.
\end{remark}

\begin{proof}
Consider the unforced closed-loop system of the plant \rfb{nlplant}
and the controller \rfb{contr} (unforced meaning $w=0$ and so $y=e$).
The state operator of the linearization of this closed-loop system at
the origin is \vspace{-2mm}
\BEQ{Ascr_cl}
  \Ascr_{cl} \m=\m \bbm{A & B \L\\ B_c C & \Phi+B_c D \L}
  \m. \vspace{-2mm}
\end{equation}
Suppose that \rfb{key_found} holds and $\sigma_p(\Phi)\subset i\rline$.
It then follows from Proposition \ref{ReWe} that there exists a
$B_c\in\rline^{\nu}$ such that $\Ascr_{cl}$ is exponentially stable.
Fix such a $B_c$. The forced closed-loop system of the plant
\rfb{nlplant} and the controller \rfb{contr} (with $B_c$ chosen as
above), with the error \rfb{error} as its output, is \vspace{-1.5mm}
$$ \bbm{\dot x \\ \dot \xi} \m=\m \Fscr(x,\xi,w) \m, \qquad e \m=\m
   h(x,\g(\xi),w) \m, \vspace{-2mm} $$
where \vspace{-1mm}
$$ \Fscr(x,\xi,w) \m=\m \bbm{f(x,\g(\xi),w)\\ \phi(\xi) + B_c
   h(x,\g(\xi),w)} \m. $$
This closed-loop system is locally exponentially stable when $w$ is
identically zero, since the linearization of $\Fscr(x,\xi,0)$ at
$(x,\xi)=(0,0)$ (its Jacobian) is $\Ascr_{cl}$.

We will now apply Lemma \ref{lm:IsBy} to show that the controller
\rfb{contr} solves the local error feedback regulator problem for the
plant \rfb{nlplant}, the exosystem \rfb{nlexo} and the error
\rfb{error}. Comparing \rfb{contr} to its general form \rfb{gencon}
we get \vspace{-1.5mm}
\BEQ{comparison}
  \eta(\xi,0) \m=\m \phi(\xi) \m, \qquad \theta(\xi) \m=\m
  \l(\xi) \FORALL \xi\in E^o \m. \vspace{-1.5mm}
\end{equation}
Let $\tau:\widetilde W^o\to E^o$ be a $C^2$ map that locally immerses
$\{W^o,s,\g\}$ into $\{E^o,\phi,\l\}$. Here $\widetilde W^o\subset
W^o$ is a open set containing zero. So $\tau(0)=0$, \vspace{-1.5mm}
\BEQ{CPDE}
 \frac{\partial\tau}{\partial w}s(w) \m=\m \phi(\tau(w)) \m,
 \qquad \g(w) \m=\m \l(\tau(w)) \FORALL w\in \widetilde W^0
 \vspace{-1.5mm} \m.
\end{equation}
By assumption, the maps $\pi:W^o\to X$ and $\g:W^o\to U$ satisfy the
nonlinear regulator equations \rfb{nlreg1} and \rfb{nlreg2}. It
follows from this, \rfb{comparison} and \rfb{CPDE} that if we define
$\sigma(w)=\tau(w)$ for all $w\in\widetilde  W^o$, then the pair $(\pi,
\sigma)$ satisfy \rfb{lm1}-\rfb{lm3} (with $\widetilde W^o$ in place of
$W^o$ in Lemma \ref{lm:IsBy}). Thus the controller \rfb{contr} solves
the above error feedback regulator problem.

If $\sigma_p(\Phi)\not\subset i\rline$, then the existence of
a local immersion $\{E^o_s,\phi_s,\l_s\}$ for $\{W^o,s,\g\}$, with
the properties described in the theorem, follows directly from
Proposition \ref{Haragus}. In this case, a $B_{cs}\in\rline^{\nu_s}$
such that the controller \rfb{contrs} solves the local error feedback
regulator problem under consideration can be found by applying the
first part of this theorem to $\{E^o_s,\phi_s,\l_s\}$ in place of
$\{E^o,\phi,\l\}$.

Finally, suppose that $\nu>0$ is the smallest integer such that for
some $\pi$ and $\g$ as in the theorem, there is a local immersion
$\{E^o,\phi,\l\}$ of order $\nu$ for $\{W^o,s,\g\}$ for which
\rfb{key_found} holds. Fix such $\pi$, $\g$ and $\{E^o,\phi,\l\}$
of order $\nu$. Our assumptions on the exosystem in \rfb{nlexo}
imply that for the dynamics of $\{W^o,s,\g\}$, the origin is a Lyapunov
stable equilibrium point and every point in $W^o$ is nonwandering.
Since $\{E^o,\phi,\l\}$ is a local immersion of the smallest order for
$\{W^o,s,\g\}$, it follows from Proposition \ref{Haragus} (see also Remark \ref{pp_verify}) that $\sigma_p(\Phi)\subset i\rline$. Hence, $B_c$ can be found so that the controller \rfb{contr} solves the local error feedback regulator problem. We claim that this controller is minimal. In the remaining part of this proof, we establish this claim by contradiction.

Assume that a controller $\CCC$ of the form \rfb{gencon}, determined
by the maps $\eta$ and $\theta$, with order $n_c<\nu$ solves the local
error feedback regulator problem. The state of this controller belongs
to an open set $X_c\subset\rline^{n_c}$ containing zero . Define $\psi
(\xi)=\eta(\xi,0)$ for all $\xi\in X_c$. It now follows from Lemma
\ref{lm:IsBy} that there exist $C^2$ maps $\tilde\pi:W^1\to X$ and
$\tilde\sigma:W^1\to X_c$, where $W^1\subset W^o$ is open and contains
zero, such that $\tilde\pi(0)=0$ and $\tilde\sigma(0)=0$ and
\rfb{lm1}-\rfb{lm3} hold (with $\tilde\pi,\m\tilde\sigma$ in place of
$\pi,\sigma$ and $W^1$ in place of $W^o$). Define $\tilde\g(w)=\theta
(\tilde\sigma(w))$ for all $w\in W^1$. Comparing \rfb{lm1} with
\rfb{nlreg1} and \rfb{lm3} with \rfb{nlreg2} it follows that $\tilde
\pi$ and $\tilde\g$ solve the regulator equations
\rfb{nlreg1}-\rfb{nlreg2}. This and \rfb{lm2} imply that $\{X_c,\psi,
\theta\}$ is a local immersion for $\{W^o,s,\tilde\g\}$. Recall the
matrices $F$, $G$ and $K$ from \rfb{shooting_in_Brussels} that
determine the linearized controller. The state operator for the
linearization of the unforced closed-loop system of the plant
\rfb{nlplant} and the controller {\bf C} at the origin is
\vspace{-2.5mm}
\BEQ{lemonade}
   \Ascr_f \m=\m \bbm{A & B K\\ G C & F+G D K} \m, \vspace{-1.5mm}
\end{equation}
which by assumption is stable. This means that the linearized
controller $(F,G,K)$ is stabilized by the linearized plant
$(A,B,C,D)$. Therefore $(K,F)$ is detectable and there is no unstable
pole/zero cancelations in the product of the plant and the controller
transfer functions. This implies that $\GGG(p)\neq0$ for each
$p\in\sigma_p(F)\cap i\rline$. We get that $\{X_c,\psi,\theta\}$ is a
local immersion of order $n_c<\nu$ for $\{W^o,s,\tilde\g\}$ for which
\rfb{key_found} holds (with $(K,F)$ in place of $(\L,\Phi)$). But this
contradicts the minimality of $\nu$ assumed in the last part of the
theorem. \end{proof}

Recall the matrices $A, P, S, C$ and $Q$ defined below
\rfb{Bolt_won_200m}. When the pair of matrices $\left(\sbm{C & Q},
\sbm{A & P\\ 0 &S}\right)$ is detectable, we have the following
corollary of Theorem \ref{newmain}.

\noindent \vspace{-6mm}
\begin{framed} \vspace{-4mm}
\begin{corollary} \label{nlDcon}
Suppose that $A$ is stable, the matrix pair $\left(\sbm{C & Q},
\sbm{A & P\\ 0 &S}\right)$ is detectable and there exist $C^2$ maps
$\pi:W^o\to X$ and $\g:W^o\to U$, where $W^o\subset W$ is an open
set containing zero, that satisfy the nonlinear regulator equations
\rfb{nlreg1} and \rfb{nlreg2}, with $\pi(0)=0$ and $\g(0)=0$. Then
there exists a $B_c\in\rline^p$ such that the controller
\vspace{-1.5mm}
\BEQ{contr_simple}
  \dot\xi \m=\m s(\xi) + B_c e \m, \qquad u \m=\m \g(\xi) \m,
  \vspace{-1.5mm}
\end{equation}
solves the local error feedback regulator problem for the plant
\rfb{nlplant}, the exosystem \rfb{nlexo} and the error \rfb{error}.
Moreover, this controller is minimal, i.e. it is of the lowest
possible order among all the controllers of the form \rfb{gencon}
that solve this error feedback regulator problem.
\end{corollary}
\vspace{-5mm} \end{framed} \vspace{-2mm}

\begin{proof}
Since $\pi$ and $\g$ satisfy the nonlinear regulator equations
\rfb{nlreg1} and \rfb{nlreg2} (and $\pi(0)=0$ and $\g(0)=0$), by
linearizing these equations at $w=0$, we get \vspace{-1.5mm}
\BEQ{lreg1}
  \Pi S \m=\m A \Pi + B\G + P \m,\qquad
  C\Pi + D\G + Q \m=\m 0 \m, \vspace{-1mm}
\end{equation}
where $\Pi=\left[\frac{\partial\pi}{\partial w}\right]_{w=0}$ and
$\G=\left[\frac{\partial\g}{\partial w}\right]_{w=0}$. Since the
the pair $\left(\sbm{C & Q},\sbm{A & P\\ 0 &S}\right)$ is
detectable, it follows from the Hautus test that \vspace{-2mm}
\BEQ{Hautus_matr}
  \Ker \bbm{A-i\alpha I & P\\ 0 & S-i\alpha I\\ C & Q}
  \m=\m \Bigg\{\bbm{0\\0}\Bigg\}\m \FORALL i\alpha\in\sigma_p(S).
  \vspace{-2mm}
\end{equation}

First we show that the pair $(\G,S)$ is detectable. Suppose that
$(\G,S)$ is not detectable. It then follows using the Hautus test
that for some $i\alpha\in\sigma_p(S)$ and some $d\in\cline^p$ with
$d\neq0$, $(S-i\alpha I)d=0$ and $\G d=0$. Using these expressions,
by multiplying both the equations in \rfb{lreg1} with $d$ (from the
right) we get that \vspace{-1.5mm}
$$ (A-i\alpha I)\Pi d + P d \m=\m 0 \m,\qquad C\Pi d + Q d \m=\m 0
   \m. \vspace{-1.5mm}$$
This means that for an $i\alpha\in\sigma_p(S)$, $\sbm{\Pi d\\d}\neq 0$
is in the kernel of the matrix on the left of \rfb{Hautus_matr}, which
contradicts \rfb{Hautus_matr}. Hence $(\G,S)$ must be a detectable
pair. This implies that the geometric multiplicity of each eigenvalue
of $S$ is 1.

Next we show that $\GGG(i\alpha)\neq0$ for all $i\alpha\in\sigma_p
(S)$. Suppose that $\GGG(i\alpha)=0$ for some $i\alpha\in\sigma_p
(S)$. Fix $d\in\cline^p$ non-zero such that $(S-i\alpha I)d=0$.
From \rfb{lreg1} we get
$$ (A-i\alpha I)\Pi d + B\G d + P d \m=\m 0 \m,\qquad C\Pi d +
   D\G d + Q d \m=\m 0   \m. \vspace{-1.5mm} $$
Solving for $\Pi d$ from the first equation above and substituting
for it in the second equation, and then using $\GGG(i\alpha)=0$, we
get that $-C(A-i\alpha I)^{-1}P d +Qd=0$. This means that $\sbm{-(A-
i\alpha I)^{-1}P d\\d}\neq0$ is in the kernel of the matrix on the
left of \rfb{Hautus_matr} for some $i\alpha\in\sigma_p(S)$, which
contradicts \rfb{Hautus_matr}. Hence $\GGG(i\alpha)\neq0$ for all
$i\alpha\in\sigma_p(S)$. Thus \rfb{key_found} holds (with $(\G,S)$
in place of $(\L,\Phi)$). It now follows from Theorem \ref{newmain}
that there exists $B_c\in\rline^p$ such that the controller
\rfb{contr_simple} solves the local error feedback regulator problem
for the plant \rfb{nlplant}, the exosystem \rfb{nlexo} and the error
\rfb{error}.

To complete the proof of this corollary, we must show that the
controller \rfb{contr_simple} is minimal. For this, according to
Theorem \ref{newmain}, it is sufficient to show that $(W^o,s,\g)$,
independently of our choice of $\pi$ and $\g$ that solve \rfb{nlreg1}
and \rfb{nlreg2}, cannot be locally immersed into an autonomous system
$(E^o,\phi,\l)$ of lower order. Suppose that $(W^o,s,\g)$ is locally
immersed in $(E^o,\phi,\l)$ via the map $\tau$, whose linearization at
the origin is $T$. Let $\Phi$ and $\L$ be linearizations of $\phi$ and
$\l$ at the origin. Then \rfb{CPDE} holds and linearizing it around
$w=0$ gives \vspace{-1.5mm}
\BEQ{sand_storm_over}
  TS \m=\m \Phi T \m, \qquad \G \m=\m \L T \m. \vspace{-1.5mm}
\end{equation}
Let $i\alpha$ be an eigenvalue of $S$ with algebraic multiplicity $k$.
Then for some $d\in\cline^p$ \vspace{-1.5mm}
\BEQ{sand_storm}
  (S-i\alpha I)^{k-1}d \m\neq\m 0\m, \qquad (S-i\alpha I)^{k}d \m
  =\m 0 \m. \vspace{-1.5mm}
\end{equation}
Since $(\G,S)$ is a detectable pair, using the Hautus test and $(S-i
\alpha I)^{k}d=0$, we get that $\G(S-i\alpha I)^{k-1}d\neq0$. The
first equation in \rfb{sand_storm_over} implies that $T(S-\mu I)^m=
(\Phi-\mu I)^m T$ for each $\mu\in\cline$ and any integer $m\geq 0$.
Hence, using \rfb{sand_storm}, we have \vspace{-1.5mm}
$$ 0 \m=\m T (S-i\alpha I)^k d \m=\m (\Phi-i\alpha I)^k T d \m,
   \qquad T (S-i\alpha I)^{k-1}d \m=\m (\Phi-i\alpha I)^{k-1} T d
   \m. \vspace{-1.5mm} $$
The second equation above, using $\G=\L T$ from \rfb{sand_storm_over}
and $\G(S-i\alpha I)^{k-1}d\neq0$, gives $\L (\Phi-i\alpha I)^{k-1}
Td\neq0$. Therefore $(\Phi-i\alpha I)^{k-1}Td\neq0$ and $(\Phi-i\alpha
I)^k Td=0$ (see the first equation above). This means that $i\alpha$
is an eigenvalue of $\Phi$ with algebraic multiplicity $m\geq k$. Thus
we have shown that if $i\alpha$ is an eigenvalue of $S$ with algebraic
multiplicity $k$, then it is also an eigenvalue of $\Phi$ with
algebraic multiplicity $m$ with $m\geq k$. Hence the order of $(E^o,
\phi,\l)$ must be larger than or equal to the order of $(W^o,s,\g)$.
Thus the controller \rfb{contr_simple} is minimal.
\end{proof}

\begin{remark} \label{contr_design}
While Theorem \ref{newmain} and Corollary \ref{nlDcon} only state the
existence of low order controllers that solve the local error feedback
regulator problem, their proofs contain a method for constructing such
controllers. In Theorem \ref{newmain}, suppose that $\sigma_p(\Phi)
\subset i\rline$. Then \rfb{contr} is the required controller in which
$\phi$ and $\l$ are known (given) and $B_c$ should be chosen so that
$\Ascr_{cl}$ in \rfb{Ascr_cl} is stable. Remark \ref{clarify} (see
also Proposition \ref{ReWe}) describes a method for choosing such a
$B_c$. When $\sigma_p(\Phi) \not\subset i\rline$, a lower order
immersion must be constructed as described in the proof of Proposition
\ref{Haragus} (this is a nontrivial task). Then \rfb{contrs} is the
required controller in which $B_c$ should be chosen so that
$\Ascr_{cl}$ in \rfb{Ascr_cl}, with $(\Phi_s,\L_s)$ in place of
$(\Phi,\L)$, is stable. Similarly in Corollary \ref{nlDcon},
\rfb{contr_simple} is the required controller in which the unknown
$B_c$ should be chosen so that $\Ascr_{cl}$ in \rfb{Ascr_cl}, with
$(S,\G)$ in place of $(\Phi,\L)$, is stable.
\end{remark}

\begin{remark} \label{relax}
In this section we have assumed that the state operator $A$ of the
linearized plant is stable. As evident from the proof of Theorem
\ref{newmain}, this assumption is used only to guarantee the
existence of a vector $B_c\in\rline^\nu$ such that $\Ascr_{cl}$ in
\rfb{Ascr_cl} is stable. Therefore, instead of the stability of
$A$, we could have directly assumed the existence of such a
vector, which is a more general assumption. For example, if
\vspace{-2mm}
$$ A \m=\m \bbm{0&1\\ 1 & -1}, \quad B \m=\m \bbm{0\\ 10}, \quad C
   \m=\m \bbm{3 & 4}, \quad D \m=\m 0,$$
$$ \Phi \m=\m \bbm{0 & 1 \\ -1 & 0}, \quad \L \m=\m \bbm{-4 & -5},$$
then $A$ is unstable but $B_c=\bbm{1\\2}$ ensures that $\Ascr_{cl}$
is stable. The problem of finding $B_c$ such that $\Ascr_{cl}$ is
stable is equivalent to the static output feedback stabilization
problem for the following linear time invariant plant:\vspace{-2mm}
$$ \bbm{\dot z_1 \\ \dot z_2} \m=\m \bbm{A^\top & 0 \\ \L^\top
   B^\top & \Phi^\top} \bbm{z_1\\z_2} + \bbm{C^\top\\ \L^\top D}
   u_c\m, \qquad y_c \m=\m z_2 \m. $$
Here $u_c(t)\in\rline$ is the input and $y_c(t)\in\rline^{\nu}$
is the output. Clearly if the static output feedback $u_c=K_c
y_c$ stabilizes the above plant, then $B_c=K_c^\top$ ensures
that $\Ascr_{cl}$ is stable. For details on the static output
feedback stabilization problem, along with some synthesis
algorithms, see \cite{CaLaSu:98,GhOuAi:97,Scherer:97,SyAbDoGr:97}
and the references therein.
\end{remark}

\begin{remark} \label{smoothness}
The claims in Theorem \ref{newmain} and Corollary \ref{nlDcon}
regarding the minimality of certain controllers are valid under
our $C^2$ smoothness assumption for all the maps in this paper.
But if we drop the requirement that the controller be described
by smooth maps, then there may exist controllers with lower
order than those we have proposed. For instance, consider the
following plant with state $x(t)\in\rline$ and input $u(t)\in
\rline$: \vspace{-2mm}
$$ \dot x(t) \m=\m -x(t) + w_1(t) + u(t)\m. $$
Here the disturbance signal $w_1(t)\in\rline$ is generated by the
second-order exosystem \vspace{-2mm}
$$ \bbm{\dot w_1(t) \\ \dot w_2(t)} \m=\m \bbm{0&1\\-1&0}\bbm{w_1
   (t) \\ w_2(t)}. $$
The tracking error is $e(t)=x(t)$, which is available for feedback.
According to Corollary \ref{nlDcon}, the order of any $C^2$ controller
that solves the local error feedback regulator problem for the above
plant, exosystem and error cannot be less than 2. However, note that
the nonlinear discontinuous static output feedback control law
$u=-\textrm{sign}\,x$ ensures that if $|w_1(0)|+|w_2(0)|<1$, then
$\lim_{t\to\infty} x(t)=0$ for all $x(0)\in\rline$. In fact, $x(t)$
converges to 0 in finite time.
\end{remark} \vspace{-7mm}

\section{Examples} \label{sec5} \vspace{-3mm}
\setcounter{equation}{0} 

\ \ \ The first example illustrates our controller design approach
using a locally stable nonlinear plant and a Lyapunov stable nonlinear
exosystem whose state trajectories are all periodic functions. An
interesting feature of this exosystem is that the state operator of
its linearization at the origin has a nontrivial Jordan block. This
example does not require the construction of an immersion. Our second
example considers a second order nonlinear plant and a second-order
linear exosystem with eigenvalues at $\pm i$. We show that no
second-order controller of the form \rfb{gencon} can solve the
regulator problem for this plant and exosystem. We then construct a
third order immersion to solve the regulator problem. In these two
examples the regulator equations can be solved by easy algebraic
manipulations. The third example is based on a practical problem in
power electronics and here the regulator equations reduce to a
quasilinear PDE of first order, without boundary conditions. This is a
challenging problem that gives us some insight into the nature of the
regulator equations. Numerical solutions to the regulator equations
have been studied using various approaches (fixed point theorems and
iterations, power series, finite elements). The relevant literature is
not large, see \cite{AguKre,AubDaPra,ByrGil,Hua:04,Rehak,RehCel} and
the references therein. Our approach in this example is
predominantly analytical.

\begin{example} \label{exmp1}
Consider the two-dimensional nonlinear exosystem \vspace{-1.5mm}
\BEQ{exmp1_exo}
  \dot w_1 \m=\m w_2-w_1^4 \m, \qquad \dot w_2 \m=\m -w_1^3 \m.
  \vspace{-1.5mm}
\end{equation}
The state operator of the linearization of this exosystem at the
origin in $\rline^2$ is $S=\sbm{0&1\\0&0}$, which is a nontrivial
Jordan block. We show that $0\in\rline^2$ is a Lyapunov stable
equilibrium point for \rfb{exmp1_exo} and that all the trajectories of
\rfb{exmp1_exo} starting sufficiently near 0 are periodic. This means
that \rfb{exmp1_exo} satisfies all the conditions required of an
exosystem, as discussed in Subsection 2.1. Existence and uniqueness of
solutions to \rfb{exmp1_exo} on maximal time intervals follows from
the standard theory of ODEs.

The phase portrait of \rfb{exmp1_exo} (disregarding the direction
along the state trajectories) is symmetric about the vertical axis.
Indeed, if $(w_1,w_2)$ is a trajectory of \rfb{exmp1_exo} on some
time interval $[0,T]$, then so is $(-\tilde w_1,\tilde w_2)$,
where $\tilde w_1(t)=w_1(T-t)$ and $\tilde w_2(t)=w_2(T-t)$ for all
$t\in[0,T]$. We claim that each state trajectory  of \rfb{exmp1_exo}
with initial condition $(0,w_2(0))$ for some $w_2(0)\in(0,1]$ is
periodic.

To establish the above claim, we consider a state trajectory
$(w_1,w_2)$ of \rfb{exmp1_exo} with initial condition $(0,w_2(0))$,
where $w_2(0)\in(0,1]$. Using the differential equations in
\rfb{exmp1_exo} we can infer the following: $w_1(t)$ increases and
$w_2(t)$ decreases along the state trajectory on some time interval
$[0,t_1]$, where $t_1$ is such that $w_1(t_1)\in(0,1)$, $w_2(t_1)
\in(0,1)$ and $w_2(t_1)=w_1^4(t_1)$. This means that $w_2(t_1)<
w_1(t_1)$. Suppose that the state trajectory exists and remains in
the open first quadrant on a maximal time interval $(0,t_2)$. Then
for all $t\in(t_1,t_2)$, $w_2(t)<w_2(t_1)$ (because $\dot w_2(t)<0$
when $w_1(t)>0$) and $w_1(t)<w_1(t_1)$ (because when $w_2(t)\in (0,
w_2(t_1))$, then $\dot w_1(t)<0$ if $w_1(t)\geq w_1(t_1)$). Hence on
the time interval $(t_1,t_2)$, $(w_1,w_2)$ is confined to the box
$\{(p,q)\m\big|\m p\in(0,w_1(t_1)),q\in(0,w_2(t_1))\}$. Along the
state trajectory in this box $\dot w_2(t)<0$ and $\dot w_2(t)<\dot
w_1(t)$. This, combined with $w_2(t_1)<w_1 (t_1)$, implies that
\vspace{-2mm}
$$ w_1(t) = w_1(t_1)+\int_{t_1}^t \dot w_1(\tau)\dd\tau \geq
   w_1(t_1)+\int_{t_1}^t \dot w_2(\tau)\dd\tau \geq w_1(t_1)-w_2
   (t_1) \quad \forall\ t\in(t_1,t_2)\m. \vspace{-1mm}$$
Hence $t_2<\infty$, $w_1(t_2)>0$ and $w_2(t_2)=0$.
Since $\dot w_1(t)<w_2(t)$ and $\dot w_2(t)<0$ whenever $(w_1(t),
w_2(t))$ is in the open fourth quadrant and since $\dot w_2(t_2)<0$,
it follows that there exists $t_3>t_2$ such that on the time interval
$(t_2,t_3)$ the state trajectory remains in the fourth quadrant,
$w_1(t_3)=0$ and $w_2(t_3)<0$. Now using the symmetry in the
phase portrait of \rfb{exmp1_exo}, we can conclude that $(w_1
(2t_3-t_2), w_2(2t_3-t_2))=(-w_1(t_2),0)$ and $(w_1(2t_3),w_2
(2t_3))=(0,w_2(0))$, i.e the state trajectory is periodic with
period $2t_3$. This completes the proof of the claim.

Next we will show that along the periodic trajectory $(w_1,w_2)$
described above, \vspace{-1.5mm}
\BEQ{lyap_cond}
  \max_{t\in[0,2t_3]} |w_1(t)| \m\leq\m (w_2(0))^{0.25} \m, \qquad
  \max_{t\in[0,2t_3]} |w_2(t)| \m=\m w_2(0) \m.  \vspace{-2mm}
\end{equation}
From the discussion in the above paragraph we get that $\max_{t
\in[0,2t_3]}|w_1(t)|\leq w_1(t_1)$, $w_1(t_1)=(w_2(t_1))^{0.25}$
and $w_2(t_1)<w_2(0)$. These expressions imply the first inequality
in \rfb{lyap_cond}. We will now show that $|w_2(t_3)|<w_2(0)$ which
implies the second equality in \rfb{lyap_cond}. Define \vspace{-3mm}
$$ z_1(t) \m=\m -w_1(t_1-t) \m, \qquad z_2(t) \m=\m w_2(t_1-t)
   \FORALL t\in[0,t_1] \m, \vspace{-2mm} $$
$$ \tilde z_1(t) \m=\m -w_1(t+t_1)\m, \qquad \tilde z_2(t) \m=\m
   -w_2(t+t_1) \FORALL t\in[0,t_3-t_1] \m. $$
Then $z_1(0)=\tilde z_1(0)$, $z_2(0)=-\tilde z_2(0)$, $z_1(t_1)=
\tilde z_1(t_3-t_1)=0$, $z_2(t_1)=w_2(0)$ and $\tilde z_2(t_3-
t_1)=-w_2(t_3)$. The functions $z_1$, $z_2$, $\tilde z_1$ and
$\tilde z_2$ satisfy the equations \vspace{-1.5mm}
$$ \dot z_1 \m=\m z_2-z_1^4 \m,\qquad \dot z_2 \m=\m -z_1^3 \m,
   \qquad \dot{\tilde{z}}_1 \m=\m \tilde z_2+\tilde z_1^4 \m,
   \qquad  \dot{\tilde{z}}_2 \m=\m -\tilde z_1^3 \vspace{-1.5mm}$$
on their intervals of definition. The curves $(z_1,z_2)$ and
$(\tilde z_1, \tilde z_2)$ (regarded as state trajectories) are in
the closed left half-plane of the phase plane. The curve $(z_1,z_2)$
starts vertically above the curve $(\tilde z_1, \tilde z_2)$ in the
phase plane and both the curves end on the vertical axis. Suppose
that $z_2(t_1)<\tilde z_2(t_3-t_1)$. Then these curves must
intersect at some point $(z_{10},z_{20})$ in the second quadrant
of the phase plane such that $z_{10}<0$ and at that point $\dd z_2/
\dd z_1\leq \dd\tilde z_2/\dd \tilde z_1$. But this inequality is
equivalent to the condition $|z_{10}|^3/(z_{20}-|z_{10}|^4)\leq
|z_{10}^3|/(z_{20}+|z_{10}|^4)$, which holds only if $z_{10}=0$
(here we have used the fact that $w_2(t)-w_1^4(t)>0$ for all
$t\in[0,t_1)$). Thus $z_2(t_1)\geq\tilde z_2(t_3-t_1)$ or
equivalently $|w_2(t_3)|<w_2(0)$.

Denote the open set enclosed by the periodic state trajectory of
\rfb{exmp1_exo} passing through the point $(0,1)$ by $\Omega$.
Fix some $(w_{10},w_{20})\in\Omega$ with $(w_{10},w_{20})\neq
(0,0)$. It then follows from \rfb{lyap_cond} that there exists a
periodic trajectory of \rfb{exmp1_exo} passing through $(0,q)$,
with $q>0$ sufficiently small, such that $(w_{10},w_{20})\notin
\Omega_1$. Here $\Omega_1$ is the open set enclosed by the state
trajectory through $(0,q)$. Therefore the trajectory $(w_1,w_2)$ of
\rfb{exmp1_exo} starting from $(w_{10},w_{20})$ lies inside the set
$\Omega$ and cannot enter the set $\Omega_1$ which contains $(0,0)$.
This and the second equation in \rfb{exmp1_exo} imply that the
state trajectory $(w_1,w_2)$ must come arbitrarily close to the
vertical axis in the phase-plane infinitely often which, along
with the first equation in \rfb{exmp1_exo}, means that $w_1$ must
change sign infinitely often and so $(w_1,w_2)$ passes through a
point $(0,\tilde w_{20})$ for some $\tilde w_{20}\in(0,1)$. From
the claim established earlier, and the uniqueness of solutions to
\rfb{exmp1_exo} backwards in time, it follows that the trajectory
$(w_1,w_2)$ is periodic.

We have shown that all the trajectories starting inside the open
set $\Omega$ containing the origin are periodic, i.e. every point $w_0
\in\Omega$ is nonwandering for the dynamics of the exosystem
\rfb{exmp1_exo}. (We remark that it is possible to show that all the
trajectories of \rfb{exmp1_exo} are periodic.) It follows from \rfb{lyap_cond}
that the origin is a Lyapunov stable equilibrium point for the exosystem
\rfb{exmp1_exo}. Below we consider an output regulation problem using
this exosystem.

Consider the nonlinear plant \vspace{-2mm}
$$ \dot x_1 \m=\m x_2 - w_1\m, \qquad \dot x_2 \m=\m -x_1-x_2
   -\sin(x_2)+(1+x_2^2)u  \vspace{-2mm}$$
with output $y=x_1$. The disturbance signal $\sbm{w_1\\ w_2}$ is
generated by the exosystem \rfb{exmp1_exo}. The control objective
is to ensure that $\lim_{t\to\infty}y(t)=0$. Hence in this example
the error is $e=x_1$ and the linearization in \rfb{plant} and
\rfb{linerror} are determined by \vspace{-2mm}
\begin{align*}
  &A \m=\m \bbm{0&1\\-1&-2}, \qquad B \m=\m \bbm{0\\1}, \qquad
  P \m=\m \bbm{-1&0\\0&0},\\
  &C \m=\m \bbm{1&0}, \qquad\qquad D \m=\m 0 \m,\qquad\qquad Q \m=\m
  \bbm{0&0}.
\end{align*}
\vspace{-13.5mm}

\begin{center}
\includegraphics[width=66mm]{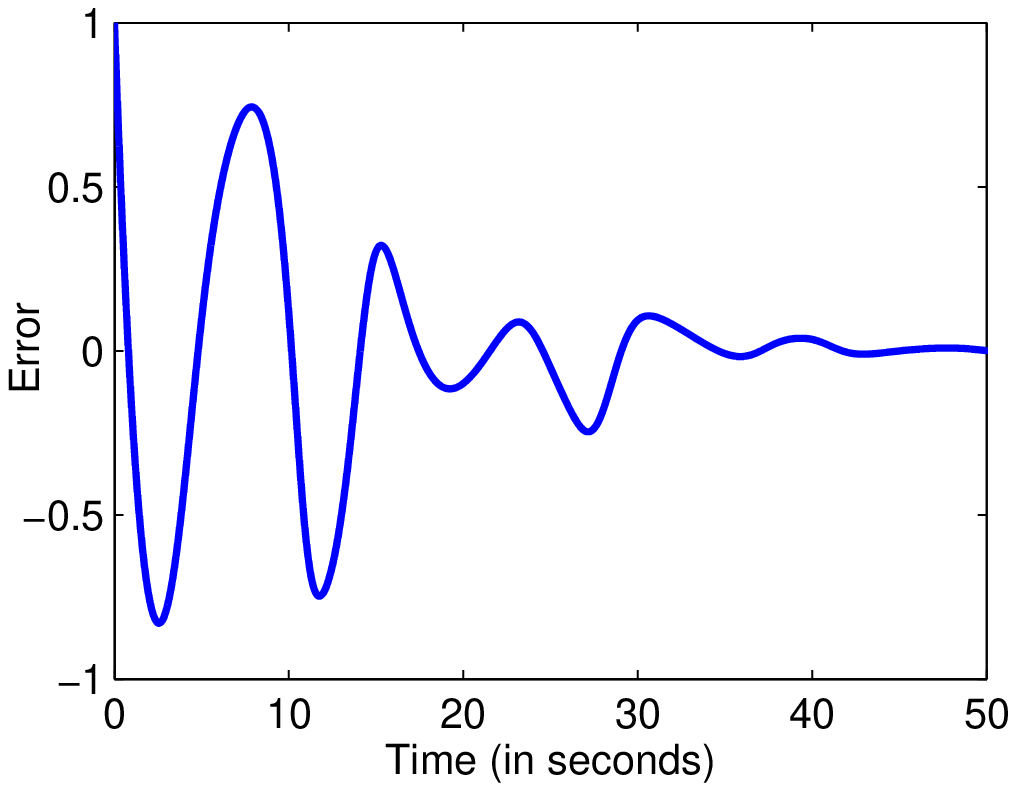}

{Figure $3$. Plot of the error (plant output) for Example 5.1}
\vspace{-2mm}
\end{center}

It is easy to verify that $A$ is stable and the pair of matrices
$\left(\sbm{C & Q},\sbm{A & P\\ 0 &S}\right)$ is detectable. The
regulator equations \rfb{nlreg1}, \rfb{nlreg2} can be solved,
because they reduce to simple algebraic equations. We get
\vspace{-1mm}
$$ \pi(w_1,w_2) \m=\m \bbm{0\\w_1} \m, \quad \g(w_1,w_2) \m=\m
   \frac{w_1+w_2-w_1^4+\sin(w_1)}{1+w_1^2} \FORALL
   w_1\m,w_2\in\rline\m. \vspace{-1mm} $$
Hence all the conditions in Corollary \ref{nlDcon} hold. The
linearization of $\g$ at $(0,0)$ is $\G=\bbm{2 & 1}$. The matrix
$\Ascr$ in \rfb{Ascr} is exponentially stable for $C_c=\G$ and
$B_c=\bbm{-0.2\\ -0.02}$. We have now completely determined a
second order controller of the form \rfb{contr_simple} which,
according to Corollary \ref{nlDcon}, is a controller of minimal
order that solves the local error feedback regulator problem for
the system under consideration. Note that in this example, the
observer-based design outlined above Remark \ref{Putin_occupies_Crimea}
would have resulted in a controller of order $n_c=6$.

We performed a simulation in MATLAB Simulink using the initial
conditions $x_1(0)=1$, $x_2(0)=-1$, $\xi_1(0)=0$, $\xi_2(0)
=0$, $w_1(0)=0.5$ and $w_2(0)=0.25$. The resulting error tends
to zero, as seen in Figure $3$.
\end{example}

\begin{example} \label{exmp2}
Consider the second order nonlinear plant \vspace{-1.5mm}
$$ \dot x_1 \m=\m x_2+x_1^2-w_1^2 \m, \qquad \dot x_2 \m=\m -x_1
   -x_2+u \vspace{-1.5mm} $$
with output $y=x_1$. The disturbance signal $w=\sbm{w_1\\ w_2}$ is
generated by the exosystem \vspace{-1.5mm}
$$ \bbm{\dot w_1\\ \dot w_2} \m=\m \bbm{0&1\\-1&0}\bbm{w_1\\w_2}
   \m. \vspace{-1mm} $$
The control objective is to ensure that $\lim_{t\to\infty} y(t)=0$,
i.e. the error is $e=x_1$. The regulator equations \rfb{nlreg1},
\rfb{nlreg2} for the above plant and exosystem reduce to algebraic
equations that have the unique solution \vspace{-1.5mm}
\BEQ{unique_g}
  \pi(w) \m=\m \bbm{0\\w_1^2} \m, \qquad \g(w) \m=\m w_1^2+2w_1w_2
  \FORALL w\in\rline^2 \m. \vspace{-1.5mm}
\end{equation}
In this example, we have \vspace{-2mm}
\begin{align*}
  & A \m=\m \bbm{0 & 1\\ -1 & -1}\m, \quad B \m=\m \bbm{0\\ 1}\m,
  \quad P \m=\m \bbm{0 & 0\\ 0 & 0}\m, \quad S\m=\m \bbm{0&1\\
  -1&0} \m,\\[4pt]
  & C \m=\m \bbm{1 & 0}\m, \quad D \m=\m 0\m, \quad Q \m=\m
  \bbm{0&0}\m.
\end{align*}
The exosystem clearly has no influence on the linearized plant or
on the error.

We claim that there exists no second order controller of the
form \rfb{gencon} that solves the local error feedback regulator
problem for the above plant and exosystem. Thus a minimal order
controller that solves this problem must be of order at least 3.
We will establish the above claim by contradiction. To this end,
we assume that the controller \rfb{gencon} with order $n_c=2$ solves
the local error feedback regulator problem. Recall the matrices
$F$, $G$ and $K$ from \rfb{shooting_in_Brussels} obtained by
linearizing the controller. Since the controller \rfb{gencon}
solves the regulator problem, its closed-loop system with the
plant is locally exponentially stable, i.e. the matrix $\Ascr_f$
in \rfb{lemonade} is stable. Lemma \ref{lm:IsBy} gives that there
exist $C^2$ maps $\pi:W^o\to\rline^2$ and $\sigma:W^o\to\rline$,
where $W^o\subset\rline^2$ is an open set containing zero, such
that $\pi(0)=0$, $\sigma(0)=0$ and \rfb{lm1}--\rfb{lm3} hold.
Since the solution to the regulator equations \rfb{nlreg1},
\rfb{nlreg2} for this example is unique and is given by
\rfb{unique_g}, comparing \rfb{lm1} with \rfb{nlreg1} and
\rfb{lm3} with \rfb{nlreg2}, we get \vspace{-1.5mm}
\BEQ{unique_soln}
  \pi(w) \m=\m \bbm{0\\w_1^2} \m,\qquad \theta(\sigma(w))
  \m=\m w_1^2+2w_1w_2 \FORALL w\in\rline^2 \m. \vspace{-1.5mm}
\end{equation}
Since $\sigma$ is a $C^2$ function, using Taylor's theorem
\cite[Theorem 10, page 179]{Mar:74}, $\sigma(w)$ and
$\frac{\partial\sigma(w)}{\partial w}$ can be expressed as
\vspace{-3.5mm}
\BEQ{sigma}
  \sigma(w) \m=\m L \bbm{w_1\\w_2}+\frac{1}{2}\bbm{a_1&b_1&2c_1
  \\ a_2&b_2&2c_2}\bbm{w_1^2\\w_2^2\\w_1w_2} + \bbm{r_1(w)\\
  r_2(w)} \m,
\end{equation}
\BEQ{psigma}
  \frac{\partial\sigma(w)}{\partial w} \m=\m L +\bbm{a_1w_1+c_1
  w_2&b_1w_2+c_1w_1\\ a_2w_1+c_2w_2&b_2w_2+c_2w_1} + \bbm{r_3(w)
  \\ r_4(w)} \m, \vspace{1mm}
\end{equation}
where $L$ is a real $2\times2$ matrix, $a_l,b_l,c_l\in\rline$ for
$l=1,2$, $r_1$ and $r_2$ are real-valued $C^2$ functions satisfying
$(|r_1(w)|+|r_2(w)|)/\|w\|^2\to 0$ as $\|w\|\to0$ and $r_3$ and
$r_4$ are real-valued $C^1$ functions satisfying $(|r_3(w)|+|r_4(w)|)
/\|w\|\to 0$ as $\|w\|\to 0$. Similarly $\theta(w)$ can be expressed
as \vspace{-1.5mm}
\BEQ{theta}
  \theta(w) \m=\m K w+R(w)\m, \vspace{-1.5mm}
\end{equation}
where $R$ is a real-valued $C^2$ function satisfying $|R(w)|/\|w\|
\to 0$ as $\|w\|\to0$.

Next we show that $L=0$. Linearizing \rfb{lm2} around the origin
gives $LS=FL$. Since $F$ is a real matrix, either $\sigma_p(S)\cap
\sigma_p(F)=\emptyset$ or $\sigma_p(S)=\sigma_p(F)$. If $\sigma_p(S)
\cap\sigma_p(F)=\emptyset$, then $LS=FL$ easily implies that $L=0$.
So suppose that $\sigma_p(S)=\sigma_p(F)$. Then $F=T^{-1}ST$ for
some invertible $2\times2$ matrix $T$ and the exponential stability
of $\Ascr_f$ implies that $K\neq0$. From $LS=FL$ and $F=T^{-1}ST$ we
get $TLS=STL$, which implies that $TL=\sbm{\alpha&-\beta\\ \beta&
\alpha}$ for some $\alpha,\beta\in\cline$. So either $L=0$ or $L$
is invertible. From the second equation in \rfb{unique_soln} we get
that \vspace{-1.5mm}
$$ \left[\frac{\partial \theta(\sigma(w))}{\partial w}\right]_{w=0}
  \m=\m K L \m=\m \bbm{0&0} \m, \vspace{-1.5mm}$$
which, along with $K\neq0$, gives that $L=0$.

Substituting for $\sigma(w)$ and $\frac{\partial\sigma(w)}{\partial
w}$ from \rfb{sigma} and \rfb{psigma} into \rfb{lm2} and equating
the quadratic terms in the resulting expression (while using $L=0$)
gives \vspace{-2.5mm}
\BEQ{Manziel_injured}
  \bbm{a_1w_1+c_1w_2 \ & b_1w_2+c_1w_1 \\ a_2w_1+c_2w_2 \ & b_2
  w_2+c_2 w_1} S \bbm{w_1\\w_2} \m=\m \frac{1}{2}\m F\bbm{a_1 & b_1 &
  2c_1\\ a_2 & b_2 & 2c_2}\bbm{w_1^2\\ w_2^2\\ w_1 w_2} \m.
  \vspace{-1.5mm}
\end{equation}
By comparing the coefficients of $w_1^2$, $w_2^2$ and $w_1w_2$ on both
sides of the above equation we get, after a simple calculation, that
\vspace{-2.5mm}
\BEQ{bolt_wins_relay}
  F\bbm{c_1\\c_2} \m=\m \bbm{a_1-b_1\\a_2-b_2} \m, \qquad F^2
  \bbm{c_1\\c_2} \m=\m -4\bbm{c_1\\c_2} \m. \vspace{-1.5mm}
\end{equation}
The second equation in \rfb{bolt_wins_relay} means that either
$\sbm{c_1\\c_2}=0$ or $F$ has eigenvalues at $\pm2i$. We will prove
below that neither of these possibilities can occur. This means
that there is no $C^2$ map $\sigma$ for which \rfb{lm2} holds for
any second order controller. This contradicts Lemma \ref{lm:IsBy}
and hence no second order controller of the form \rfb{gencon} can
solve the local error feedback regulator problem.

Suppose that $\sbm{c_1\\c_2}=0$. Then the first equation in
\rfb{bolt_wins_relay} gives $\sbm{a_1\\a_2}=\sbm{b_1\\b_2}$. Using
this, \rfb{unique_soln}, \rfb{sigma}, \rfb{theta} and $L=0$ we get
the following contradiction: \vspace{-2mm}
$$ 1 \m=\m \lim_{w_1\to 0} \frac{\theta(\sigma(w_1,0))}{w_1^2} \m=\m
   \lim_{w_1\to 0} \frac{K\sigma(w_1,0)+R(\sigma(w_1,0))}{w_1^2} \m=\m
   K\bbm{a_1\\a_2} \m, $$
$$ 0 \m=\m \lim_{w_2\to 0}\frac{\theta(\sigma(0,w_2))}{w_2^2}\m=\m
   \lim_{w_2\to 0}\frac{K\sigma(0,w_2)+R(\sigma(0,w_2))}{w_2^2} \m=\m
   K\bbm{a_1\\a_2} \m. $$
Here we have used the following fact: Denote $w_{\alpha}=\bbm{\alpha
w_1,(1-\alpha)w_2}$and $\|a\|=\sqrt{a_1^2+a_2^2}$. Then, since
$\sbm{c_1\\c_2}=0$ and $\sbm{a_1\\a_2}=\sbm{b_1\\b_2}$, for $\alpha=0$
and $\alpha=1$ \vspace{-2mm}
\BEQ{zero_limit}
  \lim_{w_{\alpha}\to 0} \frac{|R(\sigma(w_{\alpha}))|}{\|w_{\alpha}
  \|^2} \m=\m \lim_{w_{\alpha}\to 0}\frac{|R(\sigma(w_{\alpha}))|}
  {\|\sigma(w_{\alpha})\|} \frac{\|\sigma(w_{\alpha})\|}{\|w_{\alpha}
  \|^2} \m=\m \|a\| \lim_{w_{\alpha}\to 0}\frac{|R(\sigma
  (w_{\alpha}))|}{\|\sigma(w_{\alpha})\|} \m=\m 0 \m.
\end{equation}

Suppose that $F$ has eigenvalues at $\pm2i$, i.e. $F=T^{-1}\sbm{0&2
\\-2&0}T$ for some invertible real matrix $T$. We assume (without loss
of generality) that the coordinate system of the controller has been
chosen so that $F=\sbm{0&2\\-2&0}$. Substituting this $F$ into
\rfb{Manziel_injured} and comparing the coefficients of $w_1^2$,
$w_2^2$ and $w_1w_2$ on both sides of the resulting equation gives,
after a simple calculation, that $\sbm{a_1\\a_2}=\sbm{c_2\\-c_1}$ and
$\sbm{b_1\\b_2}=\sbm{-c_2\\c_1}$. Using this, \rfb{unique_soln},
\rfb{sigma}, \rfb{theta}, \rfb{zero_limit} and $L=0$ we get the
following contradiction: \vspace{-1.5mm}
$$ 1 \m=\m \lim_{w_1\to 0} \frac{\theta(\sigma(w_1,0))}{w_1^2} \m=\m
   \lim_{w_1\to 0} \frac{K\sigma(w_1,0)+R(\sigma(w_1,0))}{w_1^2} \m=\m
   K\bbm{c_2\\-c_1} \m, \vspace{-1.5mm} $$
$$ 0 \m=\m \lim_{w_2\to 0}\frac{\theta(\sigma(0,w_2))}{w_2^2}\m=\m
   \lim_{w_2\to 0}\frac{K\sigma(0,w_2)+R(\sigma(0,w_2))}{w_2^2} \m=\m
  -K\bbm{c_2\\-c_1} \m. \vspace{-1.5mm} $$

Although no second order controller can solve the local error feedback
regulator problem for the plant and the exosystem in this example, it
is easy to find a third order controller that does solve this problem.
Define the matrices $\Phi$ and $\L$ and the map $\tau:\rline^2\to
\rline^3$ as follows: \vspace{-2.5mm}
$$ \Phi \m=\m \bbm{0&0&0\\0&0&-2\\0&2&0} \m, \qquad \L \m=\m
   \bbm{0.5&1&0.5} \m, \qquad \tau(w_1,w_2) \m=\m \bbm{w_1^2+w_2^2\\
   2w_1w_2\\w_1^2-w_2^2} \m. \vspace{-2mm} $$
Using Definition \ref{immersion}, it can be verified that the
autonomous system $(\rline^2,S,\g)$ is immersed into the
autonomous system $(\rline^3,\Phi,\L)$ via the mapping $\tau$.
Clearly $(\L,\Phi)$ is detectable, $A$ is stable, $\GGG(0)\neq 0$
and $\GGG(\pm2i)\neq 0$. Hence all the conditions of Theorem
\ref{newmain} are satisfied and a third order controller of the form
\rfb{contr} with $\phi(\xi)=\Phi\xi$ and $\l(\xi)=\L\xi$ that solves
the regulator problem can be constructed. In contrast, the observer-based
controller design outlined above Remark \ref{Putin_occupies_Crimea}
would result in a controller of order $n_c=8$.
\end{example}

\begin{example} \label{exmp3}
Consider the boost converter shown in Figure 1, for which we already
gave some background in Section \ref{sec1}. We assume that the
switches are operated so fast that instantaneous values of the
currents and voltages in the switches can be replaced by their average
values over short time intervals, for instance, the sampling period,
or the time interval between two consecutive turn-on times of the
upper switch (which need not be of constant length). $q$ and $\bar q$
are complementary binary signals and $q=1$ means that the upper switch
is closed. We denote by $\Dscr$ the short-time average value of $q$, so
that $\Dscr\in[0,1]$. The state variables $z_1,z_2$ and the inputs $v$
and $i_e$ are considered practically equal to their short-time
averaged values. It is easy to derive the equations corresponding to
the averaged variables: \vspace{-2mm}
\BEQ{aver_model}
   C\dot z_1 \m=\m -i_e - \frac{z_1}{R} + \Dscr z_2 \m,\qquad
   L\dot z_2 \m=\m -rz_2 + v - \Dscr z_1 \m, \vspace{-2mm}
\end{equation}
where $C>0$ is the capacitance, $L>0$ is the inductance, $R>0$ is the
resistance of the load and $r>0$ is a small resistance that accounts
for the losses in the inductor. For a rigorous justification of this
averaging process we refer to \cite{Kassakian}, \cite{Bentsman},
\cite{MarWei}.

To bring these equations into the framework of Section \ref{sec2}, we
consider an operating point (an equilibrium state) corresponding to
the inputs $v=v_0>0$, $i_e=0$, $\Dscr=\Dscr_0\in(0,1)$. The
corresponding equilibrium state $[z_{10}\ z_{20}]^\top$ can be
computed by setting $\dot z_1=0$ and $\dot z_2=0$ in \rfb{aver_model}:
\vspace{-2mm}
$$ -\frac{z_{10}}{R} + \Dscr_0 z_{20} \m=\m 0 \m,\qquad
   -rz_{20} + v_0 - \Dscr_0 z_{10} \m=\m 0 \m, \vspace{-1mm}$$
whence \vspace{-2mm}
$$ z_{10} \m=\m \frac{\Dscr_0}{\frac{r}{R}+\Dscr_0^2} \m v_0 \m,\qquad
   z_{20} \m=\m \frac{\frac{1}{R}}{\frac{r}{R}+\Dscr_0^2} \m v_0 \m.$$
It is assumed that this is a desirable equilibrium point, i.e.
$z_{10}$ is exactly the reference output voltage. The input voltage
$v$ may deviate from $v_0$ (for instance, batteries running on low
charge), but this deviation is considered to be a constant at the
time scale of interest. The disturbance current $i_e$ is considered
to be sinusoidal, with known frequency $\alpha>0$. For instance,
this disturbance might be caused by a single-phase DC/AC converter
taking its power from this point and delivering current to some AC
load, in which case $\alpha$ would be twice the frequency of the
output voltage. Such a configuration would be typical for an
uninterruptible power supply (UPS). The state and input variables will
be the deviations of the original variables from their values at the
operating point: \vspace{-3mm}
$$ x_1 \m=\m z_1-z_{10} \m,\quad x_2 \m=\m z_2-z_{20} \m,\quad
   w_1 \m=\m v-v_0 \m,\quad u \m=\m \Dscr-\Dscr_0 \m. \vspace{-3.5mm}$$

Using \rfb{aver_model}, and the notation $w_2=i_e$, the deviations
satisfy the equations: \vspace{-1.5mm}
\BEQ{dev_eq1}
   \dot x_1 \m=\m -\frac{x_1}{RC} + \frac{\Dscr_0+u}{C} x_2 +
   \frac{z_{20}}{C} u - \frac{1}{C} w_2 \m, \vspace{-1mm}
\end{equation}
\BEQ{dev_eq2}
   \dot x_2 \m=\m -\frac{\Dscr_0+u}{L} x_1 - \frac{r}{L} x_2 -
   \frac{z_{10}}{L} u + \frac{1}{L} w_1 \m. \vspace{-1mm}
\end{equation}
The disturbance signal $w_1$ is an unknown constant while $w_2(t)=
a\cos(\alpha t+\phi)$, where $a$ and $\phi$ are unknown. We assume
that $w_1$ and $w_2$ are generated by the exosystem \vspace{-2mm}
\BEQ{Ex3_exo}
  \dot w \m=\m Sw \m, \quad w \m=\m \sbm{w_1\\w_2\\w_3} \m, \quad
  S \m=\m \sbm{0 & 0 & 0\\ 0 & 0 & \alpha \\ 0 & -\alpha & 0} \m.
  \vspace{-1mm}
\end{equation}
We want to regulate $x_1$ to zero and therefore the error is
\vspace{-3mm}
\BEQ{error_Ex3}
   e \m=\m x_1 \m. \vspace{-3mm}
\end{equation}
It is well known that it is difficult to control the higher voltage in
a boost converter because of the unstable zero dynamics (which can be
seen from the presence of a right-half plane zero in the transfer
function of the linearization from $u$ to $x_1$), see for instance
\cite{naim_97} and the references therein.

Denote $x=\sbm{x_1\\ x_2}$. We can rewrite \rfb{dev_eq1},
\rfb{dev_eq2} and \rfb{error_Ex3} in the standard form of
\rfb{nlplant} and \rfb{error}, by defining the appropriate $C^\infty$
functions $f$ and $h$ on $\rline^2\times\rline\times\rline^3$. These
functions satisfy $f(0,0,0)=0$ and $h(0,0,0)=0$. The linearization of
\rfb{dev_eq1}, \rfb{dev_eq2} and \rfb{error_Ex3} around $(0,0,0)$ is
as in \rfb{plant}--\rfb{linerror}, with \vspace{-2mm}
$$ A \m=\m \bbm{-\frac{1}{RC} & \frac{\Dscr_0}{C} \vspace{1mm}\\
   -\frac{\Dscr_0}{L} & -\frac{r}{L}}, \qquad B \m=\m \bbm{\frac
   {z_{20}}{C}\vspace{1mm} \\ -\frac{z_{10}}{L}}, \qquad P \m=\m
   \bbm{0 & -\frac{1}{C} & 0 \vspace{1mm} \\ \frac{1}{L} & 0 & 0},
   \vspace{-2mm}$$
$$ C \m=\m \bbm {1 & 0}, \qquad\qquad D \m=\m 0 \m, \qquad\qquad Q
   \m=\m \bbm{0 & 0 & 0}.$$
It is easy to see (from trace\m$A<0$ and det\m$A>0$) that $A$ is
stable. The detectability assumption contained in Corollary
\ref{nlDcon} can be verified (for any combination of parameter values)
using the Hautus test.

The nonlinear regulator equations \rfb{nlreg1}, with the notation
$\pi=\sbm{\pi^1\\ \pi^2}$, are
\begin{align*}
  \alpha\frac{\partial\pi^1}{\partial w_2}w_3-\alpha\frac{\partial
  \pi^1}{\partial w_3}w_2 &\m=\m -\frac{\pi^1}{RC} + \frac{\Dscr_0+
  \gamma}{C}\pi^2 + \frac{z_{20}}{C} \gamma - \frac{1}{C} w_2 \m,\\
  \alpha\frac{\partial\pi^2}{\partial w_2}w_3-\alpha\frac{\partial
  \pi^2}{\partial w_3}w_2 &\m=\m -\frac{\Dscr_0+\gamma}{L} \pi^1 -
  \frac{r}{L}\pi^2 - \frac{z_{10}}{L} \gamma + \frac{1}{L} w_1 \m.
\end{align*}
The regulator equation \rfb{nlreg2} is $\pi^1=0$. Using this, we
rewrite the above as
\BEQ{gamma_pi2_alg}
  0 \m=\m\frac{\Dscr_0+\gamma}{C} \pi^2 + \frac{z_{20}}{C} \gamma -
  \frac{1}{C} w_2 \m, \vspace{-1mm}
\end{equation}
\BEQ{gamma_pi2_pde}
  \alpha\frac{\partial\pi^2}{\partial w_2}w_3-\alpha\frac{\partial
  \pi^2}{\partial w_3}w_2 \m=\m - \frac{r}{L} \pi^2 - \frac{z_{10}}
  {L} \gamma + \frac{1}{L} w_1 \m.
\end{equation}
Substituting $\g$ from \rfb{gamma_pi2_pde} into \rfb{gamma_pi2_alg} we
get
\begin{align}
  &\left(\Dscr_0+\frac{L}{z_{10}}\left[-\alpha\frac{\partial\pi^2}
  {\partial w_2}w_3+\alpha\frac{\partial\pi^2}{\partial w_3}w_2-\frac
  {r}{L}\pi^2+\frac{1}{L}w_1\right]\right)\pi^2 \nonumber\\
  &+z_{20}\frac{L}{z_{10}}\left[-\alpha\frac{\partial\pi^2}{\partial
  w_2}w_3+\alpha\frac{\partial\pi^2}{\partial w_3}w_2-\frac{r}{L}\pi^2+
  \frac{1}{L}w_1\right]-w_2 \m=\m 0 \m.\label{pi2_pde}
\end{align}
This is a first order quasilinear PDE in the unknown function
$\pi^2:\rline^3\to\rline$, with no boundary conditions, only a
one-point condition: $\pi^2(0)=0$. We will solve it in a neighborhood
of $0\in\rline^3$. Our approach is to determine $\pi^2$ on circles of
the following type: $w_1$ is constant, $w_2=\rho\cos\tau$ and $w_3=
\rho\sin\tau$, where $\rho>0$ is a constant and $\tau\in[0,2\pi)$.
The motivation for our approach is that on such circles, the PDE
\rfb{pi2_pde} becomes an ODE. (In fact, these circles are the
projections of the characteristic curves of \rfb{pi2_pde} in
$\rline^4$, onto the space $\rline^3$ with coordinates $w_1$, $w_2$
and $w_3$.) For each fixed $w_1$ and $\rho$, we define a function
$\psi$ on $[0,2\pi)$ by \vspace{-2mm}
\BEQ{psi_defn}
   \psi(\tau) \m=\m \pi^2(w_1,\rho\cos\tau,\rho\sin\tau) \m.
   \vspace{-2mm}
\end{equation}
It follows by the chain rule that \vspace{-2mm}
$$ \frac{\dd\psi}{\dd\tau} \m=\m -\frac{\partial\pi^2}{\partial w_2}
   \rho\sin\tau+\frac{\partial\pi^2}{\partial w_3}\rho\cos\tau \m=\m
   -\frac{\partial\pi^2}{\partial w_2}w_3+\frac{\partial\pi^2}
   {\partial w_3}w_2 \m.$$
Using the previous expression, \rfb{pi2_pde} can be rewritten as
\BEQ{pi2_ode}
   \frac{\dd\psi}{\dd\tau} \m=\m \frac{r\psi^2+(rz_{20}-w_1-\Dscr_0
   z_{10})\psi-z_{20}w_1+z_{10}\rho\cos\tau}{\alpha L(\psi+z_{20})} \m.
\end{equation}
This ODE has to be solved for $\psi$ that satisfies the periodic
boundary condition $\psi(0)=\psi(2\pi)$. From the function $\psi$,
corresponding to different values of $w_1$ and $\rho$, we will then
construct the function $\pi^2$ using \rfb{psi_defn}.

First {\em we claim} that if $\rho<\Dscr_0z_{20}$ and
$\psi(0)>-z_{20}$, then there exists a global in time unique solution
to \rfb{pi2_ode}.  Our claim is a consequence of the following two
facts: (i) If the two inequalities mentioned above hold, then
$\psi(\tau)>-z_{20}$ for all $\tau\geq 0$ along the trajectory of
\rfb{pi2_ode}. Indeed, if this were false, then for some $t>0$
$\lim_{\tau\to t}\psi(\tau) =-z_{20}$. Let $t$ be the smallest such
number, so that $\psi(\tau)>-z_{20}$ for $\tau<t$. But such a $t$
cannot exist since the limit of $\frac{\dd\psi} {\dd\tau}$ (as given
in \rfb{pi2_ode}) for $\psi\to-z_{20}$, $\psi>-z_{20}$ is
$+\infty$. (ii) When $\psi$ is large and positive, the right side of
\rfb{pi2_ode} behaves linearly in $\psi$, so that $\psi$ cannot escape
to $+\infty$ in finite time.

In the sequel, we assume that \vspace{-1.5mm}
\BEQ{spiral}
  \rho \m<\m \beta \Dscr_0 z_{20} \quad \mbox{for some}\quad 0<\beta<1
  \m,\qquad \psi(0)>-z_{20} \m. \vspace{-1.5mm}
\end{equation}
We will solve \rfb{pi2_ode} with the periodic boundary condition under the
assumptions \rfb{spiral} and $\Dscr_0z_{10}>rz_{20}$. The latter assumption,
which is realistic since $r$ is usually very small, is needed because of the
following fact:

\noindent
{\bf Claim.} {\em If $\Dscr_0z_{10}=rz_{20}$, then there is no local
solution $\pi^2$ to \rfb{pi2_pde} which is of class $C^2$ and
satisfies $\pi^2(0)=0$.}

\noindent
Indeed, if our claim is false and such a $\pi^2$ exists, then for all values
of $w_1$ and $\rho$ sufficiently small there exists a solution $\psi$ to
\rfb{pi2_ode} that satisfies $\psi(0)=\psi(2\pi)$, obtained from $\pi^2$ via
\rfb{psi_defn}. Clearly there exists a $w_1<0$ and $\rho>0$ such that
$\psi$, as defined above, satisfies $(-\psi(0)-z_{20})w_1>\rho$. But for
this $\psi$, since $\Dscr_0z_{10}=rz_{20}$ (by assumption), it is easy
to verify that the right side of \rfb{pi2_ode} will be positive for
all $\tau\in [0,2\pi)$. This gives us the contradiction
$\psi(0)<\psi(2\pi)$, which proves our claim.

We will next identify the range of values for $w_1$ and $\rho$ for
which \rfb{pi2_ode} has a solution satisfying the periodic boundary
condition.  By solving the quadratic equation obtained from
\rfb{pi2_ode} by setting $\frac{\dd\psi}{\dd\tau}=0$ and choosing
$\cos\tau=\pm1$, we define for each $w_1$ and $\rho$ the scalars
$\psi_1$ and $\psi_2$ as follows: \vspace{-1mm}
$$ \psi_1 \m=\m \frac{-(rz_{20}-w_1-\Dscr_0z_{10})-\sqrt{(rz_{20}-w_1
   -\Dscr_0 z_{10})^2 - 4r(-z_{20}w_1+z_{10}\rho)}}{2r} \m,
   \vspace{-1mm} $$
$$ \psi_2 \m=\m \frac{-(rz_{20}-w_1-\Dscr_0z_{10})-\sqrt{(rz_{20}-w_1
   -\Dscr_0 z_{10})^2 - 4r(-z_{20}w_1-z_{10}\rho)}}{2r} \m. $$
Let $w_1^{\max} \m=\m \Dscr_0z_{10}-rz_{20}$. For each $w_1$
satisfying $|w_1|< w_1^{\max}$ let \vspace{-1mm}
$$ \rho_{w_1}^{\max} \m=\m \min\left\{\beta \Dscr_0z_{20}, \m \frac
   {(rz_{20} - w_1-\Dscr_0 z_{10})^2}{4rz_{10}}-\frac{z_{20}|w_1|}
   {z_{10}}\right\} \m. \vspace{-1mm}$$
Consider the set $\Wscr\subset\rline^2$ defined by $\Wscr=\{(w_1,\rho)
\m\big|\m |w_1|<w_1^{\max},\m 0\leq\rho<\rho_{w_1}^{\max}\}$. For each
$(w_1,\rho)\in\Wscr$ and every sufficiently small $\e>0$ the following
hold: \vspace{-3mm}
\begin{enumerate}[(1)]
 \item $\psi_1, \psi_2\in\rline$ with $\psi_1>\psi_2>-z_{20}$ and
 as $w_1$ and $\rho$ tend to zero, $\psi_1$ and $\psi_2$ tend to zero,
 \vspace{-3mm}
 \item the solution $\psi$ of \rfb{pi2_ode} with $\psi(0)=\psi_1+\e$
 satisfies $\psi(2\pi)<\psi(0)$,\vspace{-3mm}
 \item the solution $\psi$ of \rfb{pi2_ode} with $\psi(0)=\psi_2-\e$
 satisfies $\psi(2\pi)>\psi(0)$. \vspace{-2mm}
\end{enumerate}
It is easy to verify item (1) via algebraic manipulations. To see that
items (2) and (3) hold, observe that when $\psi$ is near but larger
than $\psi_1$ the right side of \rfb{pi2_ode} is negative and when it
is near but smaller than $\psi_2$ the right side of \rfb{pi2_ode} is
positive. From items (1)-(3) it follows that there exists a $\psi_0\in
[\psi_2,\psi_1]$ such that the solution of \rfb{pi2_ode} with $\psi(0)=
\psi_0$ satisfies $\psi(0)= \psi(2\pi)$. A simple algorithm for
estimating $\psi_0$, which uses the fact that no two trajectories of
\rfb{pi2_ode} can intersect in the $\psi$--$\tau$ plane, is as follows.
For $n\geq1$, let $\psi^n$ denote the solution of \rfb{pi2_ode} with
$\psi^1(0)=(\psi_1+\psi_2)/2$ and $\psi^n(0)=\psi^{n-1}(2\pi)$ for
$n>1$. It follows from items (1)-(3) that $\psi^n(0)>-z_{20}$ for all
$n$ and $\psi^n(0)\to\psi_0$ as $n\to\infty$.

For each $(w_1,\rho)\in\Wscr$, we first compute $\psi_0$ and the
solution $\psi$ to \rfb{pi2_ode} that satisfies $\psi(0)=\psi(2\pi)=
\psi_0$. Then, using \rfb{psi_defn}, we define the function $\pi^2$ on
the domain $\{(w_1,w_2,w_3)\m\big|\m |w_1|<w_1^{\max},\m 0\leq
\sqrt{w_2^2+w_3^2}<\rho_{w_1}^{\max}\}$. From the center manifold
theorem \cite[Theorem 3.22]{HaIo:11}, there exists a smooth locally
attractive invariant manifold for the dynamics of \rfb{pi2_ode} (with
$w_2$ in place of $\rho\cos\tau$) and \rfb{Ex3_exo}. Since $\pi^2$ is
constructed using periodic solutions to \rfb{pi2_ode} and
\rfb{Ex3_exo}, which lie on the above manifold for small initial data,
we can conclude that it defines a manifold which coincides with the
above manifold locally.  Hence $\pi^2$ is locally smooth. The function
$\gamma$ can be obtained using \rfb{gamma_pi2_alg} and it will also be
locally smooth.

\begin{center} \vspace{-6mm}
\includegraphics[scale=0.62]{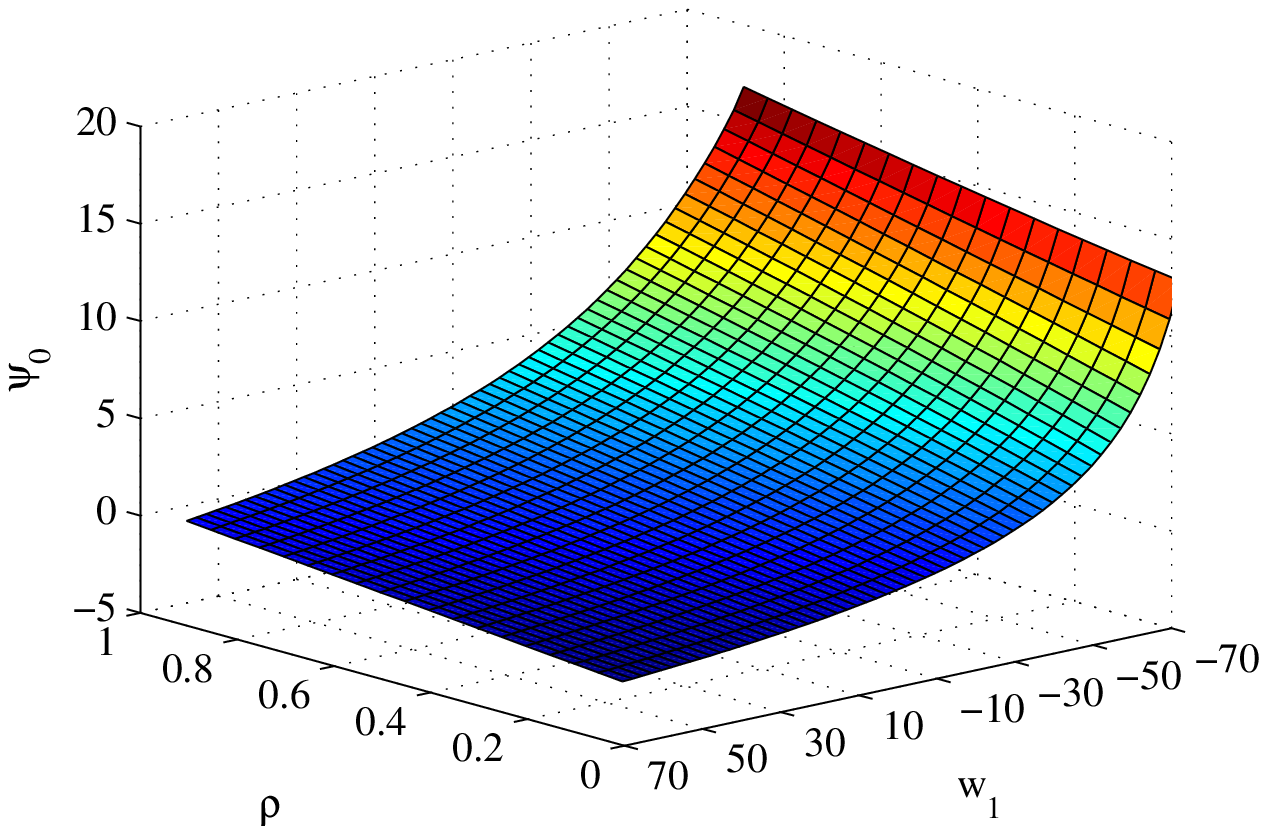} \vspace{0.5mm}

{Figure 4. Surface plot of the initial conditions $\psi_0$ (as a
function of $w_1$ and $\rho$) for which the solution $\psi$ to
\rfb{pi2_ode} satisfies $\psi_0=\psi(0)=\psi(2\pi)$.} \vspace{-2mm}
\end{center}

For our simulation we choose the nominal input voltage $v_0=100\m$V,
the desired voltage (across the load) $z_{10}=400\m$V, the load
resistance $R=400\m\Omega$, $r=0.25\m\Omega$ and $\beta=0.9$ ($\beta$
appears in \rfb{spiral}). Assuming that the switching frequency of the
switches in Figure 1 is $20\m$KHz, we choose $L=4\m$mH and $C=40\m
\mu$F to reduce the ripple in the inductor current and the voltage
across the load to reasonable values. The disturbance frequency is
$\alpha=200\pi$rad/sec (twice the nominal grid frequency). From the
above values we get that $\Dscr_0=0.2474$ and $z_{20}=4.04\m$A. We
first identify the set $\Wscr$ described earlier and compute $\psi_0$
for each $(w_1,\rho)\in\Wscr$. Figure 4 shows the surface plot of
$\psi_0$. As expected from the local smoothness of $\pi^2$, the
surface of $\psi_0$ is smooth near $(w_1,\rho)=(0,0)$. It is clear
from the plot that $\psi_0$ is a smooth function on $\Wscr$. From the
continuous dependence of solutions to \rfb{pi2_ode} on initial
conditions, we can conclude that $\pi^2$ and $\gamma$ (computed as
discussed earlier) will be smooth on their common domain of
definition. In Figure 5, we plot $\psi$ and $\gamma$ on the circles
contained in the set $\{(w_1,w_2,w_3)\m\big|\m w_1\in\{-50,0,50\},
\sqrt{w_2^2+w_3^2}=0.3\}$.

By solving \rfb{lreg1} we get $\G=\bbm{2.53\times10^{-3}& 1.06\times
10^{-4} & -2.56\times10^{-2}}$. The matrix $\Ascr$ in \rfb{Ascr} is
exponentially stable for $C_c=\G$ and $ B_c \m=\m \bbm{7.5 & -0.29
& 0.06}^{\text{T}}$. We have now completely determined a third order
controller of the form \rfb{contr_simple} which, according to Corollary
\ref{nlDcon}, is a controller of minimal order that solves the local
error feedback regulator problem for the system
\rfb{dev_eq1}--\rfb{error_Ex3}. We have performed a simulation in
Simulink with the following initial conditions: $x_1(0)=5$, $x_2(0)=0$,
$\xi_1(0)=0$, $\xi_2(0)=0$, $w_1(0)=10$, $w_2(0)=0.8$ and $w_3(0)=0$.
The resulting error decays asymptotically to zero, as seen in Figure
$6$. We remark that in this example the observer-based design outlined above
Remark \ref{Putin_occupies_Crimea} would have resulted in a controller
of order $n_c=8$.

\vspace{-1mm}
\end{example}

\begin{center} \vspace{-5mm}
\includegraphics[scale=0.58]{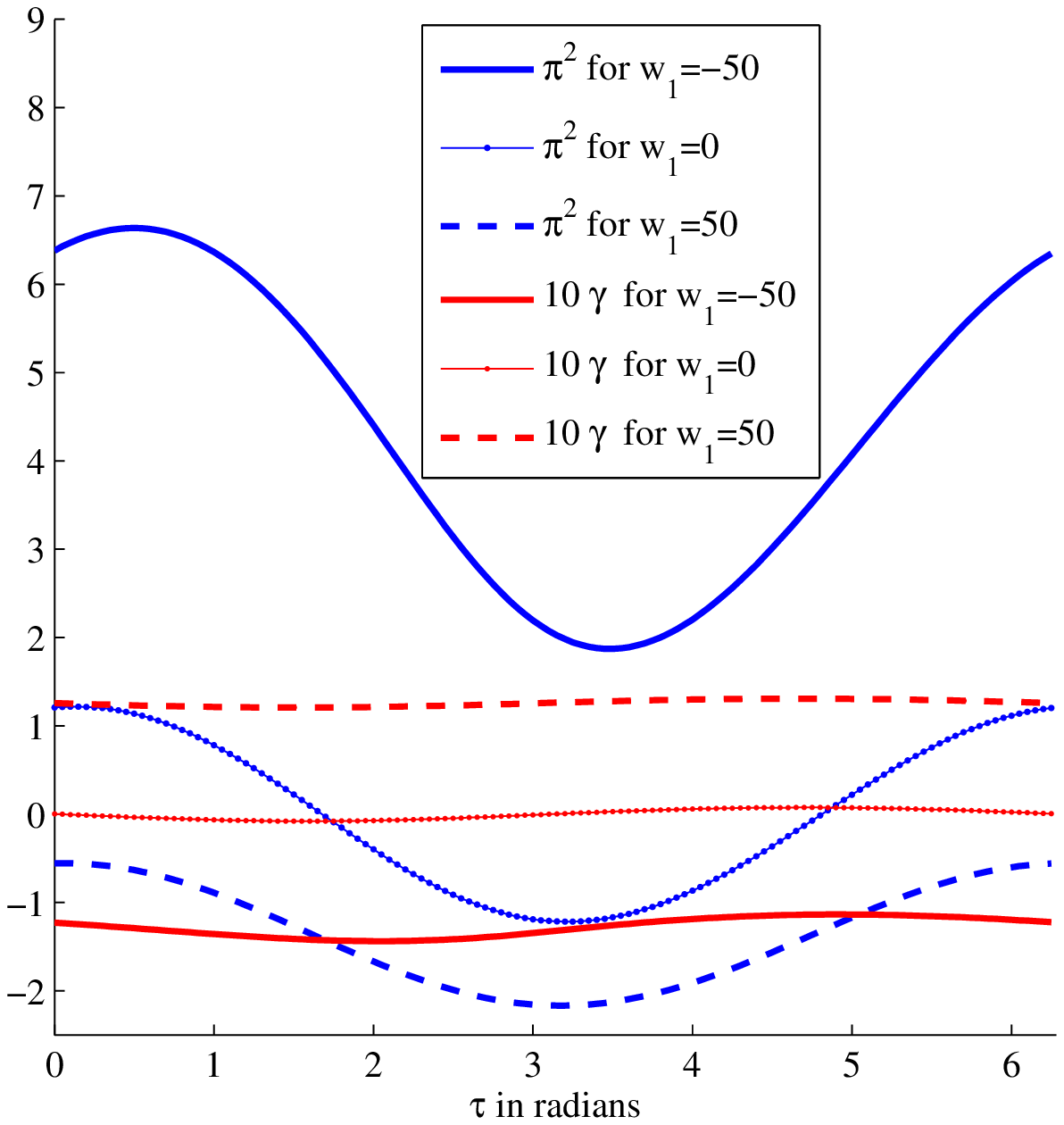} \vspace{-1mm}

{Figure $5$. Plot of $\pi^2$ and (10 times) $\gamma$ on circles of radius
$\rho=0.3$ and $w_1\in\{-50,0,50\}$, as a function of the angle $\tau$.
All the functions resemble sinusoids.}
\end{center}

\begin{center} \vspace{-2mm}
\includegraphics[scale=0.88]{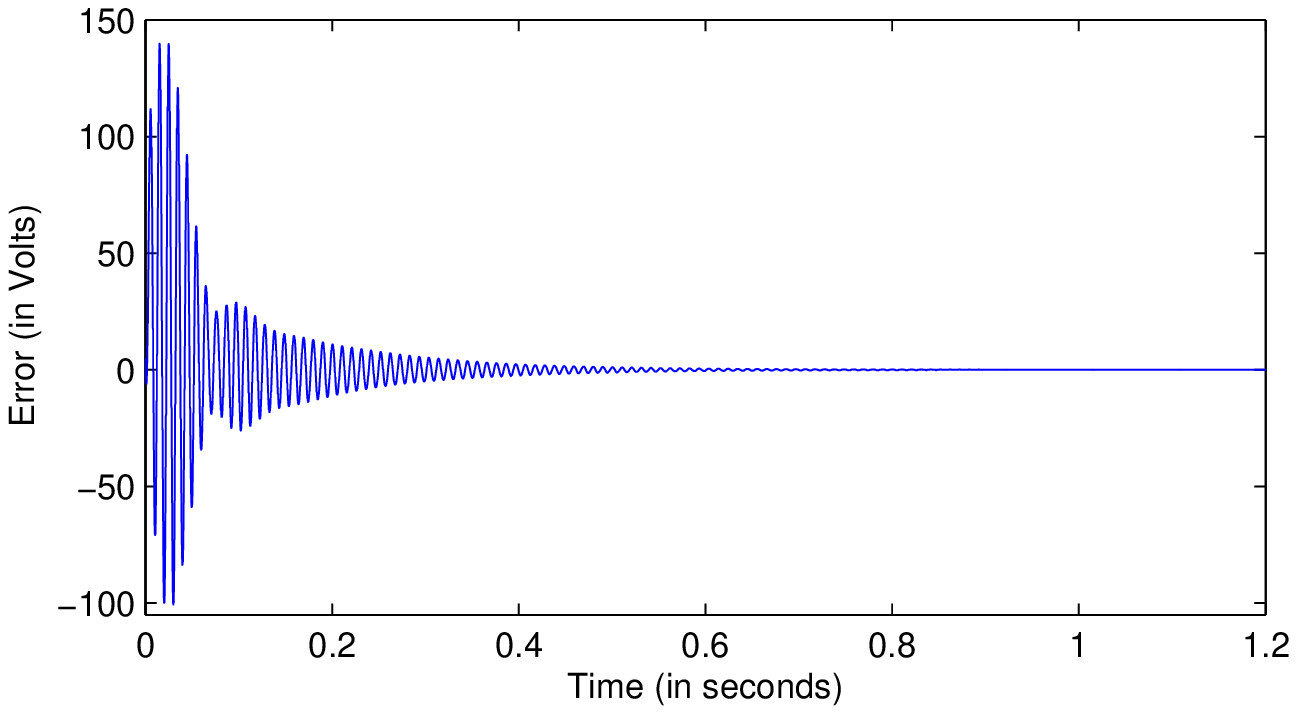} \vspace{1mm}

{Figure $6$. The error, the deviation of $z_1$ from its reference
value (400 V)}
\end{center}

\vspace{-5mm}
\section{Conclusions}\label{sec6} 
\vspace{-2mm}

\ \ \ In this work we have focussed on finding lower order controllers
that solve the local error feedback regulator problem for locally
exponentially stable nonlinear plants that are SISO from control input
to output. Under certain assumptions, we have presented a methodology
for constructing controllers whose order is equal to the order of
a detectable immersion associated with the exosystem and a solution
to the regulator equations.
Future work will address the issue of finding minimal order controllers
for multi-input multi-output plants. A relevant research area
is the construction of finite-dimensional controllers for regulating
the output of nonlinear infinite-dimensional systems (preliminary
results are in \cite{NaWe:13}). \vspace{-3mm}



\begin{thebibliography}{1} {\small
\vspace{-2mm}

\bibitem{AguKre} Aguilar, C.O., \& Krener, A.J. (2013). \m Patchy
 solution of a Francis-Byrnes-Isidori partial differential equation.
 {\em Int. J. Robust and Nonlinear Control}, {\em 23}, 1046--1061.

\bibitem{AndGev} Anderson, B.D.O., \& Gevers, M.R. (1981). \m On
 multivariable pole-zero cancellations and the stability of feedback
 systems. {\em IEEE Trans. Circuits \& Systems}, {\em 28}, 830--833.

\bibitem{AsIsMaPr:13} Astolfi, D., Isidori, A., Marconi, L., \& Praly,
 L. (2013). \m Nonlinear output regulation by post-processing internal
 model for multi-input multi-output systems. {\em Proc. of $9^{th}$
 IFAC Symposium on Nonlinear Control Systems}, Toulouse, France (pp.
 295--300).

\bibitem{AsOr:03} Astolfi, A., \& Ortega, R. (2003). \m Immersion
 and invariance: A new tool for stabilization and adaptive control
 of nonlinear systems. {\em IEEE Trans. Automatic Control}, {\em
 48}, 590--606.

\bibitem{AubDaPra} Aubin, J.P., \& Da Prato, G. (1992). \m Contingent
 solutions to the center manifold equation. {\em Annales de l'Inst.
 Henry Poincar\'e}, {\em 9}, 13--28.

\bibitem{ByrGil} Byrnes, C.I., \& Gilliam, D.S. (2007). \m Approximate
 solutions of the regulator equations for nonlinear DPS. {\em Proc. of
 the 46$^{th}$ IEEE Conf. on Decision and Control}, New Orleans, LA
 (pp. 854--859).

\bibitem{ByPrIs:97} Byrnes, C.I., Priscoli, F.D., \& Isidori, A.
 (1997). \m {\em Output regulation of uncertain nonlinear systems}.
 Boston: Birkh\"auser.

\bibitem{ByPrIsKa:97} Byrnes, C.I., Priscoli, F.D., Isidori, A., \&
 Kang, W. (1997). \m Structurally stable output regulation of
 nonlinear systems. {\em Automatica}, {\em 33}, 369--385.

\bibitem{CaLaSu:98} Cao, Y-Y., Lam, J., \& Sun, Y-X. (1998). \m
 Static output feedback stabilization: An ILMI approach.
 {\em Automatica}, {\em34}, 1641--1645.

\bibitem{Car:81} Carr, J. (1981). \m {\em Applications of Center
 Manifold Theory}. New York: Springer-Verlag.

\bibitem{ChHu:05} Chen, Z., \& Huang, J. (2005). \m Global robust
 output regulation for output feedback systems. {\em IEEE Trans.
 Automatic Control}, {\em 50}, 117--121.

\bibitem{Dav:76} Davison, E.J. (1976). \m Multivariable tuning
 regulators: the feedforward and robust control of a general
 servomechanism problem. {\em IEEE Trans. Aut. Control}, {\em
 21}, 35--47.

\bibitem{DaGo:75} Davison, E.J., \& Goldenberg, A. (1975). \m The
 robust control of a general servomechanism problem: the servo
 compensator. {\em Automatica}, {\em 11}, 461--471.

\bibitem{DeWa:78} Desoer, C.A., \& Wang, Y.T. (1978). \m On the minimum
 order of a robust servocompensator. {\em IEEE Trans. Automatic
 Control}, {\em 23}, 70--73.


\bibitem{Fra:77} Francis, B.A. (1977). \m The linear multivariable
 regulator problem. {\em SIAM J. Control and Optimization}, {\em
 15}, 486--505.

\bibitem{FrWo:75} Francis, B.A., \& Wonham, W.M. (1975). \m The
 internal model principle for linear multivariable regulators. {\em
 Appl. Math. Optimization}, {\em 2}, 170--194.

\bibitem{GhOuAi:97} Ghaoui, L.E., Oustry, F., \& AitRami, M.
 (1997). \m A cone complementarity linearization algorithm for
 static output-feedback and related problems. {\em IEEE Trans.
 Automatic Control}, {\em 42}, 1171--1176.

\bibitem{GuHo:83} Guckenheimer, J., \& Holmes, P. (1983). \m {\em
 Nonlinear Oscillations, Dynamical Systems, and Bifurcations of Vector
 Fields}. Applied Mathematical Sciences 42, New York: Springer-Verlag.

\bibitem{HaPo:00} H\"{a}m\"{a}l\"{a}inen, T., \& Pohjolainen, S.
 (2000). \m A finite-dimensional robust controller for system in the
 {CD}-algebra. {\em IEEE Trans. Automatic Control}, {\em 45},
 421--431.

\bibitem{HaIo:11} Haragus, M., \& Iooss, G. (2011). \m {\em
 Local Bifurcations, Center Manifolds, and Normal Forms in
 Infinite-Dimensionals Dynamical Systems}. Les Ulis: EDP Sciences
 and London: Springer.




\bibitem{Hua:04} Huang, J. (2004). \m {\em Nonlinear Output Regulation:
 Theory and Applications}. Philadelphia, PA: SIAM.

\bibitem{Isi:95} Isidori, A. (1995). \m {\em Nonlinear Control Systems}
 (3rd ed.). London: Springer-Verlag.

\bibitem{Isi:97} Isidori, A. (1997). \m A remark on the problem of
 semiglobal nonlinear output regulation. {\em IEEE Trans. Automatic
 Control}, {\em 42}, 1734--1738.

\bibitem{IsBy:90} Isidori, A., \& Byrnes, C.I. (1990). \m Output
 regulation of nonlinear systems. {\em IEEE Trans. Automatic
 Control}, {\em 35}, 131--140.

\bibitem{IsMa:12} Isidori, A., \& Marconi, L. (2012). \m Shifting the
 internal model from control input to controlled output in nonlinear
 output regulation. {\em Proc. of $51^{st}$ IEEE Conf. Decision and
 Control}, Maui, USA (pp. 4900--4905).


\bibitem{JaWe:08} Jayawardhana, B., \& Weiss, G. (2008). \m Tracking
 and disturbance rejection for fully actuated mechanical systems. {\em
 Automatica}, {\em 44}, 2863--2868.

\bibitem{JaWe:09} Jayawardhana, B., \& Weiss, G. (2009). \m State
 convergence of passive nonlinear systems with an $L^2$ input. {\em
 IEEE Trans. Automatic Control}, {\em 54}, 1723--1727.

\bibitem{Kassakian} Kassakian, J.G., Schlecht, M.F., \& Verghese, G.C.
 (1991). \m {\em Principles of Power Electronics}. Reading, MA:
 Addison-Wesley.

\bibitem{Kha:94} Khalil, H.K. (1994). \m Robust servomechanism output
 feedback controllers for feedback linearizable systems. {\em
 Automatica}, {\em 30}, 1587--1599.

\bibitem{Kha:02} Khalil, H.K. (2002). \m {\em Nonlinear Systems}.
 New Jersey: Prentice Hall.

\bibitem{KnIsFl:93} Knobloch, H.W., Isidori, A., \& Flockerzi, D.
 (1993). \m {\em Topics in Control Theory}. Basel: Birkh\"auser-Verlag.

\bibitem{Bentsman} Krein, P.T., Bentsman, J., Bass, R.M., \&
 Lesieutre, B.C. (1990). \m On the use of averaging for the analysis
 of power electronic systems. {\em IEEE Trans. Power Electronics},
 {\em 5}, 182--190.


\bibitem{MarWei} Margaliot, M., \& Weiss, G. (2010). \m The
 low-frequency distortion in D-class amplifiers. {\em IEEE Trans.
 Circuits and Systems II}, {\em 57}, 772--776.

\bibitem{Mar:74} Marsden, J.E. (1974). \m {\em Elementary Classical
 Analysis}. San Francisco: W.H. Freeman and Company.

\bibitem{naim_97} Naim, R.,  Weiss, G., \& Ben-Yaakov, S. (1997). \m
 $H^\infty$ control applied to boost power converters. {\em IEEE
 Trans. Power Electronics}, {\em 12}, 677--683.


\bibitem{NaWe:13} Natarajan, V., \& Weiss, G. (2013). \m Behavior of a
 stable nonlinear infinite-dimensional system under the influence of a
 nonlinear exosystem. {\em Proc. of the 1$^{st}$ IFAC Workshop on
 Control of Systems Modeled by PDEs}, Paris (pp. 155--160).

\bibitem{NaWe:14} Natarajan, V., \& Weiss, G. (2014). \m Minimal order
 controllers for output regulation of locally stable nonlinear systems.
 {\em 53rd IEEE Conf. on Decision and Control}, Los Angeles
 (pp. 4709-4714).


\bibitem{PiHu:11} Ping, Z., \& Huang, J. (2011). \m Global robust output
 regulation for a class of multivariable systems and its application to
 a motor drive system. {\em Proc. of American Control Conference}, San
 Francisco, USA (pp. 4560--4565).

\bibitem{Pri:93} Priscoli, F.D. (1993). \m Robust tracking for a class
 of nonlinear plants achieved via a linear controller. {\em Proc. of
 32$^{nd}$ IEEE Conf. Decision and Control}, San Antonio, USA
 (pp. 3550--3555).

\bibitem{Pri:97} Priscoli, F.D. (1997). \m Sufficient conditions for
 robust tracking in nonlinear systems. {\em Int. J. Control},
 {\em 67}, 825--836.

\bibitem{ReWe:03} Rebarber, R., \& Weiss, G. (2003). \m Internal model
 based tracking and disturbance rejection for stable well-posed
 systems. {\em Automatica}, {\em39}, 1555--1569.

\bibitem{Rehak} Rehak, B. (2012). \m Alternative method of solution of
 the regulator equation: $L^2$-space approach. {\em Asian Journal of
 Control}, {\em 14}, 1150--1154.

\bibitem{RehCel} Rehak, B., \& Celikovsky, S. (2008). \m Numerical
 method for the solution of the regulator equation with application to
 nonlinear tracking. {\em Automatica}, {\em 44}, 1358--1365.

\bibitem{Scherer:97} Scherer, C., Gahinet, P. \& Chilali, M. (1997).
 \m Multiobjective output-feedback control via LMI optimization. {\em
 IEEE. Trans. Automatic Control}, {\em 42}, 896--911.

\bibitem{SeIs:00} Serrani, A., \& Isidori, A. (2000). \m Global robust
 output regulation for a class of nonlinear systems. {\em Systems $\&$
 Control Letters}, {\em 39}, 133--139.

\bibitem{SeIsMa:00} Serrani, A., Isidori, A., \& Marconi, L. (2000). \m
 Semiglobal robust output regulation of minimum-phase nonlinear
 systems. {\em Int. J. Robust Nonlinear Control}, {\em 10}, 379--396.

\bibitem{SyAbDoGr:97} Syrmos, V.L., Abdallah, C.T., Dorato, P., \&
 Grigoriadis, K. (1997). \m Static output feedback - A survey. {\em
 Automatica}, {\em 33}, 125--137.

\bibitem{WeNa:14} Weiss, G., \& Natarajan, V. (2014). \m Tracking
 controller for output voltage regulation in a boost converter. {\em
 IEEE Conv. of Electr. \& Electronics Eng. in Israel}, Eilat.

\bibitem{XiDi:07} Xi, Z., \& Ding, Z. (2007). \m Global adaptive output
 regulation of a class of nonlinear systems with nonlinear exosystems.
 {\em Automatica}, {\em 43}, 143--149.

} \end{thebibliography}
\end{document}